\documentclass[12pt]{amsart}

\usepackage{amssymb, amsmath, textcomp, amsthm, amsfonts, bm}
\usepackage{bbm,dsfont}
\usepackage{color}
\usepackage{hyperref}
\usepackage[left=1.03in, right=1.03in, top=1.05in, bottom=1.05in]{geometry}  

\usepackage[mathscr]{euscript}

\usepackage{lipsum}
\usepackage{comment}

\usepackage{mathabx}

\numberwithin{equation}{section}

\makeatletter

\newcommand*\barredprod{%
  \DOTSB\mathop{%
      \@rodriguez@mathpalette \@rodriguez@overprint@bar \prod
    }\slimits@
}

\newcommand*\@rodriguez@mathpalette[2]{%
  \mathchoice
    {#1\displaystyle      \textfont         {#2}}%
    {#1\textstyle         \textfont         {#2}}%
    {#1\scriptstyle       \scriptfont       {#2}}%
    {#1\scriptscriptstyle \scriptscriptfont {#2}}%
}

\newcommand*\@rodriguez@overprint@bar[3]{%
  \sbox\z@{$#1#3$}%
  \dimen@   = \ht\z@   \advance \dimen@   \p@
  \dimen@ii = \dp\z@   \advance \dimen@ii \p@
  \dimen4 = 1.25\fontdimen 8 #2\thr@@ \relax
  \ooalign{
    \@rodriguez@bar \dimen@ \z@ \cr   
    $\m@th #1#3$\cr
    \@rodriguez@bar \z@ \dimen@ii \cr 
  }%
}
\newcommand*\@rodriguez@bar[2]{%
  \hidewidth \vrule \@width \dimen4 \@height #1\@depth #2\hidewidth
}

\makeatother

\makeatletter
\newcommand\avsuminner[2]{%
  {\sbox0{$\m@th#1\sum$}%
   \vphantom{\usebox0}%
   \ooalign{%
     \hidewidth
     \smash{\vrule height\dimexpr\ht0+1pt\relax depth\dimexpr\dp0+1pt\relax}%
     \hidewidth\cr
     $\m@th#1\sum$\cr
   }%
  }%
}
\makeatother

\newtheorem{theorem}{Theorem}[section]
\newtheorem{lemma}[theorem]{Lemma}

\newtheorem{proposition}[theorem]{Proposition}

\newtheorem{remark}[theorem]{Remark}
\newtheorem*{remark*}{Remark}

\newtheorem{definition}[theorem]{Definition}
\newtheorem*{definition*}{Definition}
\newtheorem{corollary}[theorem]{Corollary}

\newcommand{\al}{\alpha}
\newcommand{\be}{\beta}

\newcommand{\de}{\delta}

\newcommand{\De}{\Delta}
\newcommand{\e}{\varepsilon}

\newcommand{\ka}{\kappa}

\newcommand{\si}{\sigma}
\newcommand{\Si}{\Sigma}

\newcommand{\cs}{\mathcal S}

\newcommand{\cp}{\mathcal P}

\newcommand{\cb}{\mathcal B}

\newcommand{\cd}{\mathcal D}

\newcommand{\cn}{\mathcal N}

\newcommand{\sa}{\mathscr A}

\newcommand{\wt}{\widetilde}

\newcommand{\ZR}{\mathbb{R}}

\newcommand{\ZN}{\mathbb{N}}

\newcommand{\ZS}{\mathbb{S}}

\newcommand{\Id}{{\rm{\bf{1}}}}

\newcommand{\cB}{{\mathcal B}}
\newcommand{\cT}{{\mathcal T}}
\newcommand{\cQ}{{\mathcal Q}}
\newcommand{\cC}{{\mathcal C}}
\newcommand{\cg}{\mathcal G}
\newcommand{\cl}{\mathcal L}

\newcommand{\dist}{{\rm dist}}

\newcommand{\supp}{\rm{supp}}

\newcommand{\hdim}{\dim_{\rm H}}

\usepackage{graphicx}

\begin{document}

\title[Study guide]{A study guide for ``On the Hausdorff dimension of Furstenberg sets and orthogonal projections in the plane"
\\[1ex]  After T. Orponen and P. Shmerkin}

\author{Jacob B. Fiedler}
\address{Department of Mathematics, University of Wisconsin, Madison, WI, USA}
	\email{jbfiedler2@wisc.edu}

\author{Guo-Dong Hong}
\address{Department of Mathematics, California Institute of Technology, Pasadena, CA,
USA}
	\email{ghong@caltech.edu}

\author{Donggeun Ryou}
\address{Department of Mathematics, University of Rochester, Rochester, NY, USA}
	\email{dryou@ur.rochester.edu}

\author{Shukun Wu}
\address{Department of Mathematics, Indiana University Bloomington, Bloomington, IN, USA}
	\email{shukwu@iu.edu}

\begin{abstract}
    This article is a study guide for ``On the Hausdorff dimension of Furstenberg sets and orthogonal projections in the plane" by Orponen and Shmerkin \cite{OS21}. We begin by introducing Furstenberg set problem and exceptional set of projections and provide a summary of the proof with the core ideas.
\end{abstract}

\maketitle

\tableofcontents

\section{Introduction}
This paper aims to give an idea of the proof of the main theorem in \cite{OS21}, which studies Furstenberg sets and the exceptional set of orthogonal projections.
First of all, let us start with the \textit{$(s,t)$-Furstenberg sets}:

\begin{definition}
    A set $F \in \mathbb{R}^2$ is called an \textit{$(s,t)$-Furstenberg sets} if there exist $\cl$, a family of lines $\ell(a,b) := \{(x,y) : y=ax+b\}$ with $\hdim\cl:=\hdim\{(a,b): \ell(a,b) \in \cl\} \geq t$, such that $\hdim(F \cap \ell) \geq s$ for all $\ell \in \cl$.
\end{definition}

The main theorems of the paper are as follows.
\begin{theorem}\label{MainT_Furst}
    For every $s \in (0,1)$ and $t \in (s,2]$, there exists $\epsilon=\epsilon(s,t)>0$ such that given any \textit{$(s,t)$-Furstenberg sets} $F\in \mathbb{R}^2$, we have $\hdim F \geq 2s+\epsilon$. 
\end{theorem}
Next, we switch to the orthogonal projection: for $e \in \mathbb{S}^1$ we denote $\pi_e:\mathbb{R}^2 \rightarrow \mathbb{R}$ by the orthogonal projection to the line passing through the origin spanned by $e$, i.e., $\pi_e(x) = e \cdot x$.
\begin{theorem}\label{MainT_proj}
    Consider $s\in (0,1)$ and $t \in (s,2]$, then there exists $\epsilon= \epsilon(s,t)>0$, such that for any $K \subset \mathbb{R}^2$ with $\hdim(K)=t$, 
    \[
        \hdim\{ e \in \mathbb{S}^1 \,|\, \hdim \Pi_e(K) < s\} \leq s-\epsilon.
    \]
\end{theorem}
Both theorem \ref{MainT_Furst} and theorem \ref{MainT_proj} follow from their discretized version: Theorem \ref{MainT_tube}, which we will introduce later. Before that, we provide some background on Furstenberg sets and exceptional sets of projections.

\subsection{Furstenberg set problem}
The Furstenberg set problem is a fractal analog of the Kakeya problem in $\mathbb{R}^2$. A \textit{Kakeya set} in $\mathbb{R}^n$ is a compact set $K \subset \mathbb{R}^n$ which contains a unit line segment in every direction, and the \textit{Kakeya set problem} asks the smallest possible Hausdorff dimension of a Kakeya set. The conjecture is that the Hausdorff dimension is $\geq n$. It is true when $n=2$, but it is still open in higher dimensions, i.e. $n \geq 3$. Similarly, the \textit{$(s,t)$-Furstenberg set problem} asks the smallest possible Hausdorff dimension of an $(s,t)$-Furstenberg set.

As a special case, if we consider a set $F\subset \mathbb{R}^2$ such that for all directions $e \in S^1$, there is a line $\ell_e$ in the direction $e$ such that $\hdim(F \cap \ell_e) \geq s$, then we call it \textit{$s$-Furstenberg set} and the corresponding problem is called the \textit{$s$-Furstenberg set} problem. Wolff showed a lower bound for the $s$-Furstenberg set problem in \cite{Wolff99}, by using the Szemer\'edi-Trotter theorem, which can be generalized to the $(s,t)$-Furstenberg set problem as follows.
\begin{theorem}
    Consider $s\in (0,1)$ and $t \in (s,1]$, then for any $(s,t)$-Furstenberg set $F$, we have $\hdim F\geq \max\{\frac{t}{2}+s ,2s\}$.
\end{theorem}
This is essentially the ``elementary'' bound for the problem. It was conjectured that every $(s,t)$-Furstenberg set $F \subset \mathbb{R}^2$ has Hausdorff dimension
\begin{equation*}
    \hdim F \geq \min \{ s+t, \frac{3s+t}{2}, s+1\}.
\end{equation*}
In recent work, Ren and Wang \cite{RW23} fully resolved this conjecture, which we will detail towards the end of this section. 

An example for $\hdim F = s+t $ or $s+1$ can be constructed by using ``Cantor target," and an example for $\hdim F = \frac{3s+t}{2} $ can be found in \cite{Wolff99}. Only the case when $t=1$ is described in \cite{Wolff99}, but it can be extended to $t \leq 1$. Note that $\frac{3s+t}{2} \geq s+1$ if $t \geq 1$.
\begin{figure}[ht]
\centering
\includegraphics[width=0.3\textwidth]{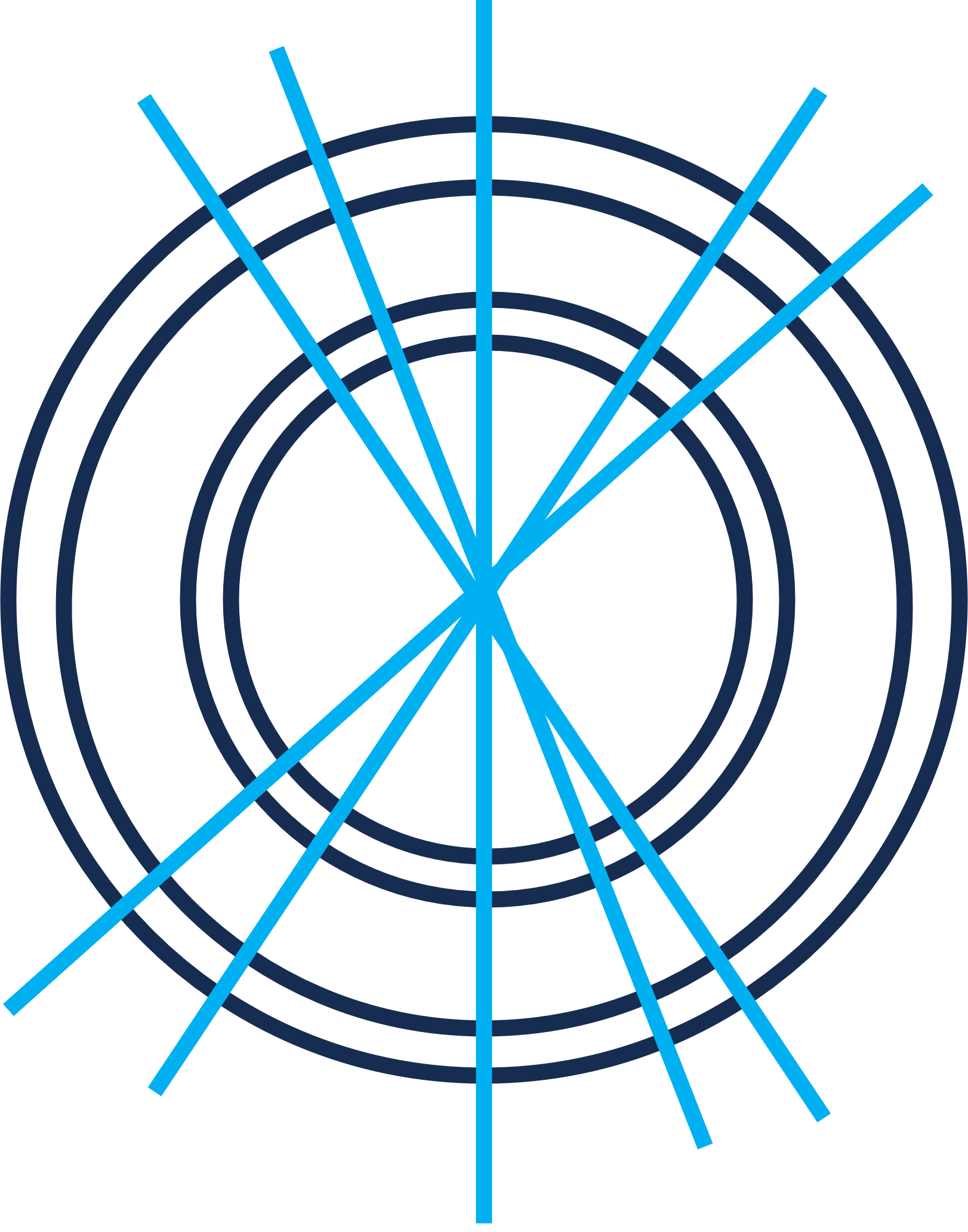}
\caption{Cantor target: Let $\cC$ be circles such that the set of radii is a Cantor set of Hausdorff dimension $s$ and let $\cl$ be a set of lines pass through origin such that Hausdorff dimension of the set of angles is $t\leq 1$. Let $K=\cup_{c \in \cC, l \in \cl} C \cap l$, then $K$ is a $(s,t)$-Furstenberg set of dimension $s+t$.}
\end{figure}

\subsection{Exceptional set of projections}
For simplicity, let us focus on the case when $K \subset \mathbb{R}^2$ with $\hdim K \in (0,1]$. Marstrand's projection theorem \cite{Mars54} says that if $\hdim K  =s \leq 1$, then $\hdim \pi_e(K) =s$ for $\mathcal{H}^1$ almost all $e \in \mathbb{S}^1$. Then, we can ask how small the exceptional directions are such that $\hdim \pi_e (K) <s$. 

Let $\hdim K =t$. Some of the known results on this question are as follows: 
    \begin{itemize}
        \item \cite{Kauf68} For any $s \in (0,t]$,
        \begin{equation}\label{Kauf_est}
        \hdim \{ e \in \mathbb{S}^1 \,|\, \hdim \pi_e(K) < s\} \leq s.    
        \end{equation}
        \item \cite{Ober12}
        \begin{equation}\label{Ober_est}
        \hdim\{ e \in \mathbb{S}^1 \,|\, \hdim \pi_e(K) < t/2 \} =0.    
        \end{equation}
        \item \cite{Bourgain-projection} Given $\epsilon>0$, there exists $\delta>0$, such that
        \[
        \hdim\{ e \in \mathbb{S}^1 \,|\, \hdim \pi_e(K) < t/2 + \delta\} < \epsilon.
        \]
        In particular, 
         \[
        \hdim \{ e \in \mathbb{S}^1 \,|\, \hdim \pi_e(K) \leq t/2 \} =0.
        \]
    \end{itemize}
Let $\hdim K =t$, then the conjecture was that for any $s \in [0, t]$,
\begin{equation}\label{Proj_conj}
    \hdim\{ e \in \mathbb{S}^1 \,|\, \hdim  \pi_e(K) < s\} \leq \max \{ 2s-t, 0\}.
\end{equation}
Ren and Wang's work \cite{RW23} also implies that this conjecture is true . 

\subsection{History of progress on the Furstenberg set problem}    

Here, we will cover many of the results on the Furstenberg set problem, leading up to its recent resolution. In particular, we explain the key role played by Orponen and Shmerkin's work \cite{OS21}. As mentioned, Wolff was the first to record progress on the Furstenberg set problem in \cite{Wolff99}. He proved the following bound on the size of $s$-Furstenberg sets $F$, which he indicated likely originated with Furstenberg and Katznelson
\begin{equation*}
    \hdim (F)\geq \max \{2s, \frac{1}{2} + s\}.
\end{equation*}
\noindent In 2003, the work of Bourgain \cite{Bou03} in conjunction with the prior work of Katz and Tao \cite{KT01} led to an $\epsilon$-improvement in the above bound around $s=\frac{1}{2}$. However, for $s$ appreciably larger than $\frac{1}{2}$, Wolff's bounds remained the strongest for some time. 

Prior to the work of Orponen and Shmerkin that is the subject of this study guide, other progress on the Furstenberg set problem was made. In 2010, Molter and Rela \cite{MR10} generalized $s$-Furstenberg sets to $(s, t)$-Furstenberg sets. They established the bound
\begin{equation*}
\hdim(F)\geq \max\{s + \frac{t}{2}, 2s + t-1\}.
\end{equation*}
\noindent In 2017, N. Lutz and Stull \cite{LS20} improved this bound for certain values of $s$ and $t$ . They obtained 
\begin{equation*}
    \hdim(F)\geq s + \min\{s, t\}.
\end{equation*}
\noindent Notably, this was the first application of the point-to-set principle of J. Lutz and N. Lutz \cite{LL18} to establish a new result in classical fractal geometry. Other relevant work includes Hera, Shmerkin, and Yavicoli's $\epsilon$-improvement for $(s, 2s)$-Furstenberg sets \cite{HSY22}; Benedetto and Zahl's quantified improvement over this bound \cite{BZ21}; and separate $\epsilon$-improvements on the \emph{packing} dimension bound for $s$-Furstenberg sets by Orponen \cite{Orp20} and Shmerkin \cite{Shm22}.

Considering the above work, Theorem \ref{MainT_Furst} is notable because it establishes an $\epsilon$-improvement to Hausdorff dimension bound for $s$-Furstenberg sets for all $s>\frac{1}{2}$, not just $s$ near $\frac{1}{2}$. Additionally, it yields an improvement to the $(s, t)$-Furstenberg set problem for $t>s$, and not just $t=2s$, marking the first general improvement to the elementary bounds. Even beyond these advances, however, \cite{OS21} was a breakthrough. This paper played a key role in later work that eventually led to a resolution of the Furstenberg set problem and sharp bounds for the dimension of exceptional sets of projections, which we now outline.

This $\epsilon$-improvement was a crucial ingredient of Orponen, Shmerkin, and Wang's recent sharp bounds on the size of radial projections in  $\mathbb{R}^2$ \cite{OSW24}. They establish these bounds by via a bootstrapping scheme that depends on the small but sufficiently uniform improvement to the Furstenberg set problem of \cite{OS21}. The study guide \cite{BBMO24} is a helpful tool for reading Orponen, Shmerkin, and Wang's paper. 

Next, in work from 2023, Orponen and Shmerkin were able to use this radial projection theorem to establish a major explicit improvement to the $(s, t)$-Furstenberg set problem, as well improved bounds for exceptional sets of projections \cite{OS23}. In particular, for an $s$-Furstenberg set $F$, they showed
\begin{equation}\label{OS2023 general bound}
    \hdim(F)\geq \max\{2s +\frac{(1-s)^2}{2-s}, 1+s\}.
\end{equation}
This bound follows from a resolution of the Furstenberg set conjecture under the assumption that the set of lines is ``regular'' (meaning the set looks roughly the same at all scales). The idea is more or less that the regularity gives a quasi-product structure to these sets, enabling the application of a discretized sum-product theorem; this theorem is the first of \cite{OS23} and its proof uses the radial projection theorem of \cite{OSW24}. 

However, the resolution for regular Furstenberg sets is not only interesting because it led to the bound \eqref{OS2023 general bound}. It also resolves one of the ``enemy'' scenarios for the Furstenberg set problem, related to the aforementioned sum-product phenomenon. The other enemy scenario comes from ``well-spaced'' sets, the opposite of regular sets in the sense that well-spaced sets look as \emph{different} as possible across scales, while regular sets look very similar across scales. Progress was made on this other enemy through a sharp incidence estimate for well-spaced tubes by Guth, Solomon, and Wang \cite{GSW19}. Ren and Wang generalized this sharp incidence estimate to a wider class of semi-well spaced sets \cite{RW23}. Then, they combined the extremal cases of regular sets and semi-well spaced sets, establishing for $(s, t)$-Furstenberg sets $F$ that
\begin{equation*}
    \hdim F \geq \min \{ s+t, \frac{3s+t}{2}, s+1\}.
\end{equation*}
This is the sharp bound. The same work also gives the conjectured exceptional set bounds, namely that if $\hdim K =t$, then for any $s \in [0, t]$,
\begin{equation*}
    \hdim\{ e \in \mathbb{S}^1 \,|\, \hdim  \pi_e(K) < s\} \leq \max \{ 2s-t, 0\}.
\end{equation*}

\subsection{Roadmap of \texorpdfstring{\cite{OS21}}{OS21}}
The paper \cite{OS21} can be divided into three parts:
\begin{enumerate}
    \item Part 1: Section 2,3\\
    In the first part, the goal is to introduce the ``discretized" Szemeredi-Trotter estimate and see how to deduce the bound on the Furstenberg set problem and the exceptional set of orthogonal projection from these discretized estimates.
        \begin{itemize}
        \item Section 2: Szemeredi-Trotter theorem and its discretized version. See \eqref{Section 3}.
        \item Section 3: Proof of Furstenberg set problem and exceptional set of orthogonal projection by using discretized Szemeredi-Trotter estimate. See \eqref{Section 2}.
        \end{itemize}
    \item Part 2: Section 4,5,6 and Appendix\\
    In the second part, the goal is to prove an improved discretized Szemeredi-Trotter estimate under regular conditions with the help of an induction on scales scheme.
        \begin{itemize}
            \item Section 4,5: An induction on scale scheme.
            \eqref{Section 4}.
            \item Appendix: Discretized Szemeredi-Trotter estimate (a dichotomy under regular conditions).
            \item Section 6: Improved discretized Szemeredi-Trotter estimate under regular conditions.
            \eqref{est_w_reg}.
        \end{itemize}
    \item Part 3: Section 7,8,9\\
    In the final part, the goal is to introduce a multiscale analysis to get an improved discretized Szemeredi-Trotter estimate by ``interpolating" the original estimate and the improved estimate under regular conditions.
        \begin{itemize}
            \item Section 7:
            Multiscale analysis 1: Combining incidence estimates from different scales. See \eqref{Section 5}.
            \item Section 8: 
            Multiscale analysis 2: Choosing good decompositions of different scales. See \eqref{Section 5}.
            \item Section 9: 
            Finishing the proof of the improved discretized Szemeredi-Trotter estimate. See \eqref{est_w_reg}.
        \end{itemize}
\end{enumerate}

\section{Heuristic proof} \label{Section 2}
In this section, we will see how to use the Szemer\'edi-Trotter theorem to give a heuristic estimate of both the Furstenberg set problem and the exceptional set of the orthogonal projection, which motivated further study of the discretized Szemer\'edi-Trotter in the next section.

We denote by $A \lesssim B$ if $A \leq CB$ for some constant $C>0$, and we abbreviate $A \lesssim B \lesssim A$ to $A \sim B$. If the implicit constant depends on parameters like $\epsilon$, we denote $A \lesssim_\epsilon B$. We will sometimes use $A \lessapprox B$ to denote $A \lesssim_\epsilon \delta^{-\epsilon} B$ and use $A \approx B$ to write $A \lessapprox B$ and $B \lessapprox A$.

\subsection{Szemer\'edi-Trotter theorem: sharp and elementary versions}
\begin{theorem} Let $\mathcal{L}$ be a finite set of lines in $\mathbb{R}^2$ and $\mathcal{P}$ be a finite set of points in $\mathbb{R}^2$. Let $\mathcal{I}(\mathcal{P}, \mathcal{L})$ is the set of incidences between lines in $\mathcal{L}$ and points in $\mathcal{P}$, i.e. $\mathcal{I}(\mathcal{P},\mathcal{L}):=\{(p,L)\in \mathcal{P} \times \mathcal{L}: p \in L \}$. Then, we have the following estimates.
    \begin{itemize}
        \item (sharp version)
        \[
        |\mathcal{I}(\mathcal{P},\mathcal{L})|
        \lesssim
        |\mathcal{L}|^{2/3}|\mathcal{P}|^{2/3}+|\mathcal{L}|+|\mathcal{P}|.
        \]
        \item (elementary version 1)
        \[
        |\mathcal{I}(\mathcal{P},\mathcal{L})|
        \lesssim
        |\mathcal{L}|^{1/2}|\mathcal{P}|+|\mathcal{L}|.
        \]
        \item (elementary version 2)
        \[
        |\mathcal{I}(\mathcal{P},\mathcal{L})|
        \lesssim
        |\mathcal{L}||\mathcal{P}|^{1/2}+|\mathcal{P}|.
        \]
    \end{itemize}
\end{theorem}

\subsection{Heuristic proof for Furstenberg set problem} 
The main result in this subsection is to see how the exponents in the Furstenberg set problem come heuristically from Szemer\'edi-Trotter theorem:
\begin{itemize}
    \item Sharp version corresponds to the conjectural exponent $\frac{3s+t}{2}$.
    \item Elementary version 1 corresponds to the exponent $s+\frac{t}{2}$.
    \item Elementary version 2 corresponds to the exponent $2s$.
\end{itemize}

Since the ideas are similar, we will just show the first case that the sharp version corresponds to the conjectural exponent $\frac{3s+t}{2}$. The first heuristic we will use is that we treat any $\delta$-balls as a point, and any $\delta$-tube as a line. 
\footnote{$\delta$-ball here means the $\delta$ neighborhood of a point, and $\delta$-tube means the $\delta$ neighborhood of a line segment. One can find a more precise definition in Section 3. }
The second heuristic we will use is that we treat the Hausdorff dimension as the box dimension. 

After discretizing the $(s,t)$-Furstenberg set in the $\delta$-scale, we have $\sim \delta^{-t}$ many $\delta$-tubes and each $\delta$-tubes contain $\sim \delta^{-s}$ many $\delta$-balls. Therefore, we have the following lower bound for the incidence:
\[
    \delta^{-s} \cdot |\mathcal{L}|
    \lesssim |\mathcal{I}(\mathcal{P},\mathcal{L})|.
\]

However, by the sharp version of the Szemer\'edi-Trotter theorem, we have
\[
    |\mathcal{I}(\mathcal{P},\mathcal{L})|
    \lesssim
    |\mathcal{L}|^{2/3}|\mathcal{P}|^{2/3}+|\mathcal{L}|+|\mathcal{P}|.
\]

If we do a case-by-case study, we will find $|\mathcal{L}|^{2/3}|\mathcal{P}|^{2/3}$ is the dominant term for the upper bound of the incidence. Therefore, we have
\[
    \delta^{-s} \cdot |\mathcal{L}| \lesssim
    |\mathcal{L}|^{2/3}|\mathcal{P}|^{2/3}
\]
and combing that $|\mathcal{L}| \sim \delta^{-t}$ gives us $|\mathcal{P}| \gtrsim \delta^{-\frac{3s+t}{2}}$, which essentially says $dim_{\mathcal{H}}F\geq \frac{3s+t}{2}$.

\subsection{Heuristic proof for the exceptional set of the orthogonal projection}
The main result in this subsection is to see how the exponents in the exceptional set of orthogonal projection come heuristically from Szemer\'edi-Trotter theorem:
\begin{itemize}
    \item Sharp version corresponds to Oberlin's conjecture \eqref{Proj_conj}.
    \item Elementary version 1 corresponds to Oberlin's estimate \eqref{Ober_est}.
    \item Elementary version 2 corresponds to Kaufman's estimate \eqref{Kauf_est}.
\end{itemize}
Again, since the ideas are similar, we will only show how sharp version corresponds to Oberlin's conjecture. Now let $K \subset \mathbb{R}^2$ with $dim_{\mathcal{H}}K=t$. Denote the exceptional set $\Theta=\{ e \in \mathbb{S}^1 \,|\, \dim_{\mathcal{H}} \pi_e(K) < s\}$ and set $\alpha=dim_{\mathcal{H}}\Theta$. Our goal is to show that $\alpha \leq 2s-t$.

Now suppose $e \in \Theta$ is an exceptional direction, then the discretization of $K$ at $\delta$-scale can be covered by $\lesssim \delta^{-s}$ many parallel $\delta$-tubes with the direction perpendicular to $e$. Since there are $\delta^{-\alpha}$ exceptional direction, the total amount of tubes $|\mathcal{L}|\lesssim \delta^{-s-\alpha}$. Since $dim_{\mathcal{H}}K=t$, $|\mathcal{P}| \sim \delta^{-t}$. Finally, since every $\delta$-ball is incident to $\sim \delta^{-\alpha}$ $\delta$-tubes coming from each exceptional direction, so we have the lower bound for the incidence:
\[
    \delta^{-\alpha} \cdot |\mathcal{P}|
    \lesssim |\mathcal{I}(\mathcal{P},\mathcal{L})|.
\]

Again, by the sharp version of the Szemer\'edi-Trotter theorem with the dominant term $|\mathcal{L}|^{2/3}|\mathcal{P}|^{2/3}$ by case-by-case study, we have
\[
    \delta^{-\alpha} \cdot |\mathcal{P}|
    \lesssim
    |\mathcal{I}(\mathcal{P},\mathcal{L})|
    \lesssim
    |\mathcal{L}|^{2/3}|\mathcal{P}|^{2/3}.
\]
Combined with $|\mathcal{L}|\lesssim \delta^{-s-\alpha}$ and $|\mathcal{P}| \sim \delta^{-t}$, we then have
\[
    \delta^{-\alpha} \cdot \delta^{-t}
    \lesssim
    \delta^{-2(s+\alpha)/3}\cdot \delta^{-2t/3},
\]
which is equivalent to $\alpha \leq 2s-t$, and we are done.
\begin{remark}
    We should mention that it is easier to make the second heuristic rigorous by pigeonholing arguments, while making the first heuristic corresponding to the sharp version of the Szemer\'edi-Trotter theorem rigorous is much more difficult. However, we can make the first heuristic rigorous in the case corresponding to the elementary version, and this is the main subject of the next section.
\end{remark}

\begin{remark}
    We encourage the reader to read Section 3 in \cite{OS21} to see how to use the pigeonholing argument to reduce the Hausdorff dimension statement to the box dimension statement (or discretized statement) where our heuristic can be applied.
\end{remark}

\section{Discretized Szemer\'edi-Trotter theorem} \label{Section 3}

\subsection{Elementary bound for the Szemer\'edi-Trotter theorem}
Here we give an elementary bound of the Szemer\'edi-Trotter theorem.
\begin{theorem}[Elementary bound for Szemer\'edi-Trotter theorem] 
    \[
    |\mathcal{I}(\mathcal{P},\mathcal{L})|
    \lesssim
    |\mathcal{L}|^{1/2}|\mathcal{P}|+|\mathcal{L}|.
    \]
    \[
    |\mathcal{I}(\mathcal{P},\mathcal{L})|
    \lesssim
    |\mathcal{L}||\mathcal{P}|^{1/2}+|\mathcal{P}|.
    \]
\end{theorem}
\begin{proof}
First of all, by point line duality, the above two inequalities are essentially equivalent (by changing the role of point and line). Therefore, it suffices to prove the first version.\\
Note that $|\mathcal{I}(\mathcal{P},\mathcal{L})|=\sum_{l\in\mathcal{L}}\sum_{p\in\mathcal{P}} \chi_{l}(p)$, where $\chi_{l}(p)=1$ if $p\in l$ and $\chi_{l}(p)=0$ if $p\notin l$.

By Cauchy-Schwarz, we have
\[
    \sum_{l\in\mathcal{L}}\sum_{p\in\mathcal{P}} \chi_{l}(p) \leq
    (\sum_{l\in\mathcal{L}}1)^{1/2}
    (\sum_{l\in\mathcal{L}}\sum_{p,p'\in\mathcal{P}}\chi_{l}(p)\chi_{l}(p'))^{1/2}.
\]

Observe that the first term $\sum_{l\in\mathcal{L}}1=|\mathcal{L}|$ and the second term can be split into the diagonal part ($p=p'$) and the off-diagonal part ($p\neq p'$). \\For the diagonal part,
\[
    \sum_{l\in\mathcal{L}}\sum_{p=p'\in\mathcal{P}}\chi_{l}(p)\chi_{l}(p')
    =\sum_{l\in\mathcal{L}}\sum_{p\in\mathcal{P}} \chi_{l}(p)=
    |\mathcal{I}(\mathcal{P},\mathcal{L})|.
\]
For the off-diagonal part, the key idea is that at most one line can pass through any given pair of two distinct points ($\sum_{l\in\mathcal{L}}\chi_{l}(p)\chi_{l}(p')\leq1$ when $p\neq p'$):
\[
    \sum_{l\in\mathcal{L}}\sum_{p\neq p'\in\mathcal{P}}\chi_{l}(p)\chi_{l}(p')
    =\sum_{p\neq p'\in\mathcal{P}}\sum_{l\in\mathcal{L}}\chi_{l}(p)\chi_{l}(p') \leq
    \sum_{p\neq p'\in\mathcal{P}}1
    \thicksim
    |\mathcal{P}|^2.
\]
To sum up, what we have done is as follows:
\[
    |\mathcal{I}(\mathcal{P},\mathcal{L})|^2
    \lesssim
    |\mathcal{L}|
    (|\mathcal{I}(\mathcal{P},\mathcal{L})|+|\mathcal{P}|^2)
\]

Finally, note that either $|\mathcal{I}(\mathcal{P},\mathcal{L})|^2
    \lesssim
    |\mathcal{L}|
    |\mathcal{I}(\mathcal{P},\mathcal{L})|$
or $|\mathcal{I}(\mathcal{P},\mathcal{L})|^2
    \lesssim
    |\mathcal{L}|
    |\mathcal{P}|^2$, which corresponds to either $|\mathcal{I}(\mathcal{P},\mathcal{L})|
    \lesssim
    |\mathcal{L}|$ 
or $|\mathcal{I}(\mathcal{P},\mathcal{L})|
    \lesssim
    |\mathcal{L}|^{1/2}
    |\mathcal{P}|$.
Summing up these two cases, we get the desired bound.

\end{proof}

\subsection{Elementary bound for the discretized Szemer\'edi-Trotter theorem}
As we see, the key geometric consideration in the elementary Szemer\'edi-Trotter theorem is that, at most, one line can pass through any given pair of two distinct points. However, this is not the case in the discretized setting. 

Let $\delta \in 2^{-\mathbb{N}}$ be a dyadic number and let $\mathcal{D}_\delta( \mathbb{R}^2) $ be the family of half-open dyadic $\delta$-cubes in $\mathbb{R}^2$. If $A $ is a set in $\mathbb{R}^2$, we denote $\mathcal{D}_\delta (A) $ the family of $p \in \mathcal{D}_\delta(\mathbb{R}^2)$ such that $A \cap p \neq \emptyset$ and denote $|A|_\delta  = |\mathcal{D}_\delta(A)|$. If $A = [0,1)^2$, we abbreviate $ D_\delta([0,1)^2)$ to $D_\delta$.

Similarly, we also define $\delta$-tubes. For $q \in \mathcal{D}_\delta(\mathbb{R}^2)$, a $\delta$-tube $T(q)$ is a set of form
\[
\bigcup_{(a,b) \in q}\{ (x,y) \in \mathbb{R}^2 : y=ax+b \}.
\]
We define the family of $\delta$-tubes $\mathcal{T}^\delta$ by $\{ T(q) : q \in \mathcal{D}_\delta([-1,1] \times \mathbb{R}) \} $.

Now, let us give a definition of the discretized set first.
\begin{definition}[$(\de,s,C)$-set]
\label{def-1}
Let $P\subset\ZR^d$ be a bounded set, $d\geq1$. Let $\de>0$ and let $0\leq s\leq d$ and $C>0$. We say that $P$ is a $(\de,s,C)$-set if for any $r$-cube $Q$
\begin{equation*}
    |P\cap Q|_\de\leq Cr^s|P|_\de.
\end{equation*}
Similarly, let $\mathcal{L}$ be a set of lines in $\mathbb{R}^2$. we say that $\mathcal{L}$ is a $(\delta, s, C)$ set if the set of points $(a,b)$ such that $\{(x,y) \in \mathbb{R}^2 : y=ax+b \} \in \mathcal{L}$ is a $(\delta ,s ,C)$-set. 

\end{definition}

Now assume that $\mathcal{P}$ is a $(\delta,t)$-set
\footnote{We usually denote a $(\de,s,C)$-set as $(\delta,s)$-set if the constant $C$ is not important.}
and $\forall p \in \mathcal{P}$, there exists a $(\delta,s)$-set of tubes $ \mathcal{T}(p)$ such that $p \in T$ for all $T \in \mathcal{T}(p)$. We set $\mathcal{T}=\bigcup_{p \in \mathcal{P}} \mathcal{T}(p)$. Define the incidence $\mathcal{I}(\mathcal{P},\mathcal{T})=\{(p,T)\in \mathcal{P} \times \mathcal{T}:T \in \mathcal{T}(p)\}$. Now let us study that given two distinct $\delta$ cubes $p$ and $p'$, how many $\delta$-tube can pass through them: $\sum_{T\in\mathcal{T}}\chi_{T}(p)\chi_{T}(p')$.

Observe that if a $\delta$-tube passes two distinct $\delta$ cubes, then its slope can only range from an interval of size $\thicksim \frac{\delta}{d(p,p')}$. Moreover, since $\mathcal{T}(p)$ is a $(\delta,s)$-set, the amount of tubes that can pass through $p$ and $p'$ are smaller than $|\mathcal{T}(p)| \cdot (\frac{\delta}{d(p,p')})^{s}$. To sum up, 
\[
    \sum_{T\in\mathcal{T}}\chi_{T}(p)\chi_{T}(p') \lesssim
    |\mathcal{T}(p)| \cdot \left(\frac{\delta}{d(p,p')}\right)^{s}.
\]
Therefore, the off-diagonal part becomes
\[
    \sum_{T\in\mathcal{T}}\sum_{p\neq p'\in\mathcal{P}}\chi_{T}(p)\chi_{T}(p')= 
    \sum_{p\in\mathcal{P}} \sum_{p'\in\mathcal{P}: p'\neq p}
    \sum_{T\in\mathcal{T}}\chi_{T}(p)\chi_{T}(p') \lesssim
    \sum_{p\in\mathcal{P}} \sum_{p'\in\mathcal{P}: p'\neq p}
    |\mathcal{T}(p)| \cdot \left(\frac{\delta}{d(p,p')}\right)^{s}.
\]
Since $\mathcal{P}$ is a $(\delta,t)$-set, we have
\[
    \sum_{p'\in\mathcal{P}: p'\neq p}
    \left(\frac{1}{d(p,p')}\right)^{s} \thicksim
    \sum_{j}|\{p'\in\mathcal{P}:d(p,p')\sim2^{-j}\}|\cdot
    2^{js} 
    \lesssim
    \sum_{j} |\mathcal{P}|\cdot 2^{j(s-t)}
    \lesssim |\mathcal{P}|,
\]
where in the last step, we use the assumption that $s<t$. Therefore, we have
\[
    \sum_{T\in\mathcal{T}}\sum_{p\neq p'\in\mathcal{P}}\chi_{T}(p)\chi_{T}(p')
    \lesssim |\mathcal{P}| \cdot \delta^s
    \sum_{p\in\mathcal{P}}
    |\mathcal{T}(p)|.
\]

In particular, if we assume that $|\mathcal{T}(p)|\thicksim M$ for some constant $M$ (which can be done by passing to a refinement of $\mathcal{P}$ after pigeonholing), then the contribution of the off-diagonal part becomes
\[
    \sum_{T\in\mathcal{T}}\sum_{p\neq p'\in\mathcal{P}}\chi_{T}(p)\chi_{T}(p')
    \lessapprox_{\delta}
    |\mathcal{P}|^2 \cdot \delta^s \cdot M.
\]

Now we record what we have done below:
\begin{proposition}
\label{discretized-ST}
    Let $0< s \leq t$.
    \footnote{Note that in the argument above, we require $s<t$ to illustrate the idea. However, the argument can be extended to the case when $s\leq t$ with more involved analysis. We encourage readers to read Proposition 2.13 in \cite{OS21}.}
    Assume that $\mathcal{P}$ is a $(\delta,t)$-set and $\forall p \in \mathcal{P}$, there exists a $(\delta,s)$-set of tubes $\mathcal{T}(p)$ such that $p \in T$ for all $T \in \mathcal{T}(p)$. Set $\mathcal{T}=\bigcup_{p \in \mathcal{P}} \mathcal{T}(p)$. 
    Also assume that $|\mathcal{T}(p)|\thicksim M$ for some constant $M$. Then 
    \[
    |\mathcal{I}(\mathcal{P},\mathcal{T})|
    \lessapprox_{\delta}
    M\cdot \delta^s |\mathcal{T}|^{1/2} |\mathcal{P}|+|\mathcal{T}|.
    \]
\end{proposition}

As a result, we have $M\cdot |\mathcal{P}| \leq  |\mathcal{I}(\mathcal{P},\mathcal{T})|\lessapprox_{\delta}
    M\cdot \delta^s |\mathcal{T}|^{1/2} |\mathcal{P}|$ (the first term on the right-hand side dominates), which is equivalent to that $\delta^{-2s}\lessapprox_{\delta}|\mathcal{T}|$:
    and this gives us essentially $dim_{\mathcal{H}}(F) \geq 2s$, where $F$ is an $(s,t)$-Furstenberg set.

Finally, with a little extra effort, we can modify our discretized incidence estimate to the following more complicated version, which is better for our induction on scale scheme later (Corollary 2.14 in the original paper \cite{OS21}):

\begin{corollary}\label{cor_T_est}
    Let $0 \leq s \leq t \leq 1$, and let $C_p,C_T \geq 1$. Let $\cp \subset \cd_\delta$ be a $(\delta, t,C_p)$-set. Assume that for every $p \in \cp$ there exist a $(\delta, s, C_T)$-set $\cT(p) \subset \cT^\delta$ of dyadic $\delta$-tubes with the properties that $\cT \cap p \neq \emptyset$ for all $\cT \in \cT(p)$, and $|\cT(p) | \sim M$ for some $M \geq 1$. Then,
    \begin{equation*}
        |T| \gtrapprox_\delta (C_PC_T)^{-1}M\delta^{-s} (M\delta^s)^{\frac{t-s}{1-s}}.
    \end{equation*}
\end{corollary}

For the rest of the paper, we will prove the following theorem, one of the main theorems in \cite{OS21}. Theorem \ref{MainT_Furst} and \ref{MainT_proj} follow from this theorem.

\begin{theorem}\label{MainT_tube}
    For $s \in (0,1)$ and $t \in (s,2]$, there exists $\epsilon(s,t) >0$ such that the following holds for all small enough $\delta \in 2^{-\mathbb{N}}$ depending only on $s$ and $t$. Let $\cp \subset \cd_\delta$ be a $(\delta, t, \delta^{-\epsilon})$-set with $\cup \cp \subset [0,1)^2$, and $\cT \subset \cT^\delta$ be a family of dyadic $\delta$-tubes. Assume that for every $p \in \cp$, there exists a $(\delta, s, \delta^{-\epsilon})$-set $\cT(p) \subset \cT$ such that $T \cap p \neq \emptyset$ for all $T \in \cT(p)$. Then $|\cT| \geq \delta^{-2s-\epsilon}$.
\end{theorem}

\section{The induction on scales scheme} \label{Section 4}

The main goal of this part of the study guide is to explain the proof of an induction on scales-type proposition, which corresponds to sections 4 and 5 of \cite{OS21}. The proposition is 5.2 in the original paper, and we will need it for two purposes: proving an improved incidence estimate at the ``regular" scales, and when we perform our multi-scale decomposition. The point is that it allows us to relate information at different scales. First, we make the following definition:

\begin{definition}
A $(\delta, s, C, M)$-nice configuration is a pair $(\mathcal{P}_0, \mathcal{T}_0)\subset \mathcal{D}_\delta \times \mathcal{T}^\delta$ such that for every $p\in\mathcal{P}_0$, there is some family $\mathcal{T}(p)\subset\mathcal{T}_0$ (with each tube intersecting $p$) that is a $(\delta, s, C)$-set and has cardinality $M$. 
\end{definition}

\noindent We can also define $S_Q$ to be the homothety taking the square $Q$ to the unit square, which will be useful as we think about different scales. We state the (rather lengthy) proposition here and will discuss the meaning of the various parts after.

\begin{proposition}\label{prop_2scale}
    Fix dyadic numbers $0<\delta<\Delta\leq1$. Let $(\mathcal{P}_0, \mathcal{T}_0)$ be a $(\delta, s, C_1, M)$-nice configuration. Then there exist sets $\mathcal{P}\subset\mathcal{P}_0$ and $\mathcal{T}(p)\subset\mathcal{T}_0(p)$, $p\in\mathcal{P}$ such that, denoting $\mathcal{T}=\bigcup_{p\in\mathcal{P}}\mathcal{T}(p)$ the following hold:
\begin{enumerate}
    \item $\vert\mathcal{D}_\Delta(\mathcal{P})\vert\approx_\delta\vert\mathcal{D}_\Delta(\mathcal{P}_0)\vert$ and $\vert\mathcal{P}\cap Q\vert \approx_\delta \vert\mathcal{P}_0\cap Q\vert$ for all $Q\in\mathcal{D}_\Delta(\mathcal{P})$.
    \item $\vert\mathcal{T}(p)\vert\gtrapprox_\delta\vert\mathcal{T}_0(p)\vert=M$ for $p\in\mathcal{P}$.
    \item There are $\mathcal{T}_{\Delta}\subset \mathcal{T}^\Delta$, $C_{\Delta } \approx_\delta C_1$ and $M_\Delta \geq 1$ such that  $(\mathcal{D}_\Delta(\mathcal{P}), \mathcal{T}^\Delta(\mathcal{T}))$ is $(\Delta, s, C_\Delta, M_\Delta)$-nice for some $C_\Delta\approx_\delta C_1$ and $M_\Delta\geq 1$.
    \item For each $Q \in \mathcal{D}_\Delta(\mathcal{P})$ there exist $C_Q\approx_\delta C_1$, $M_Q\geq1$ and a family of tubes $\mathcal{T}_Q\subset \mathcal{T}^{\delta/\Delta}$ such that $(S_Q(\mathcal{P}\cap Q), \mathcal{T}_Q)$ is $(\delta/\Delta, s, C_Q, M_Q)$ nice.   
\end{enumerate}
Furthermore, the families $\mathcal{T}_Q$ can be chosen so that 
\begin{equation}\label{prop_2scale_eq}
    \dfrac{\vert\mathcal{T}_0\vert}{M}\gtrapprox_\delta\dfrac{\vert\mathcal{T}^\Delta(\mathcal{T})\vert}{M_\Delta}\cdot \left( \max_{Q\in\mathcal{D}_\Delta(\mathcal{P})}\dfrac{\vert\mathcal{T}_Q\vert}{M_Q}\right).
\end{equation}
\end{proposition}

First, in what sense is this an induction on scales-type proposition? The key is the last inequality. Here, $\mathcal{T}_0$ is a family of tubes at a scale of $\delta$, $\mathcal{T}^\Delta(\mathcal{T})$ is a family of tubes at a scale of $\Delta$, and each $\mathcal{T}_Q$ is a family of tubes at a scale of $\delta/\Delta$. So, we essentially get a lower bound on incidences at one scale in terms of incidences at two other scales.

Properties (1) and (2) essentially guarantee that we don't have to remove much from our set of squares and our set of tubes in order to achieve the refinements with the desired properties. Note that if we did remove a significant amount, the main inequality above would not be particularly useful. Properties (3) and (4) guarantee that the sets of squares and tubes we obtain at the scales $\delta$ and $\delta/\Delta$ have the similarly nice properties compared to the original configuration, with reasonable constants.

To prove this proposition, we will first need a crucial lemma (Proposition 4.1 in the original paper \cite{OS21}).

\begin{lemma}\label{Pigeonholing before induction}
Let $0<\delta\leq\Delta\leq 1$ be dyadic numbers and let $C_1, M>1$. Let $\mathcal{P}$ be a finite set and assume that for every $p\in\mathcal{P}$, there is an associated $(\delta, s, C_1)$ set $\mathcal{T}(p)\subset\mathcal{T}^{\delta}$ of cardinality between $\frac{M}{2}$ and $M$ with each tube intersecting $[0, 1)^2$. Then, there is a subset $\bar{\mathcal{P}}\subset\mathcal{P}$ of cardinality $\approx_\Delta\vert \mathcal{P}\vert$ and a collection of tubes $\bar{\mathcal{T}}_{\Delta}\subset\mathcal{T}^{\Delta}$ each intersecting $[0, 1)^2$ such that 
\begin{enumerate}
    \item $\bar{\mathcal{T}}_{\Delta}$ is a $(\Delta, s, C_2)$ set with $C_2\lessapprox_\Delta C_1$.
    \item There is a constant $H\approx_\Delta M\cdot \vert\mathcal{P}\vert/\vert\bar{\mathcal{T}}_{\Delta}\vert$ such that for any $\mathbf{T}\in\bar{\mathcal{T}}_{\Delta}$,
    \begin{equation*}
    \vert\{(p, T)\in\bar{\mathcal{P}}\times \mathcal{T}^\delta: T\in\mathcal{T}(P) \text{ and } T\subset\mathbf{T}\}\vert\gtrsim H.
    \end{equation*}
\end{enumerate}
\end{lemma}

The idea of this lemma is to provide a way of covering a collection of small tubes with a suitably separated collection of larger tubes. The reason we need a relatively complicated statement is that the naive version of this proposition is not even true, as illustrated below:

\begin{center}

\includegraphics[scale=.5]{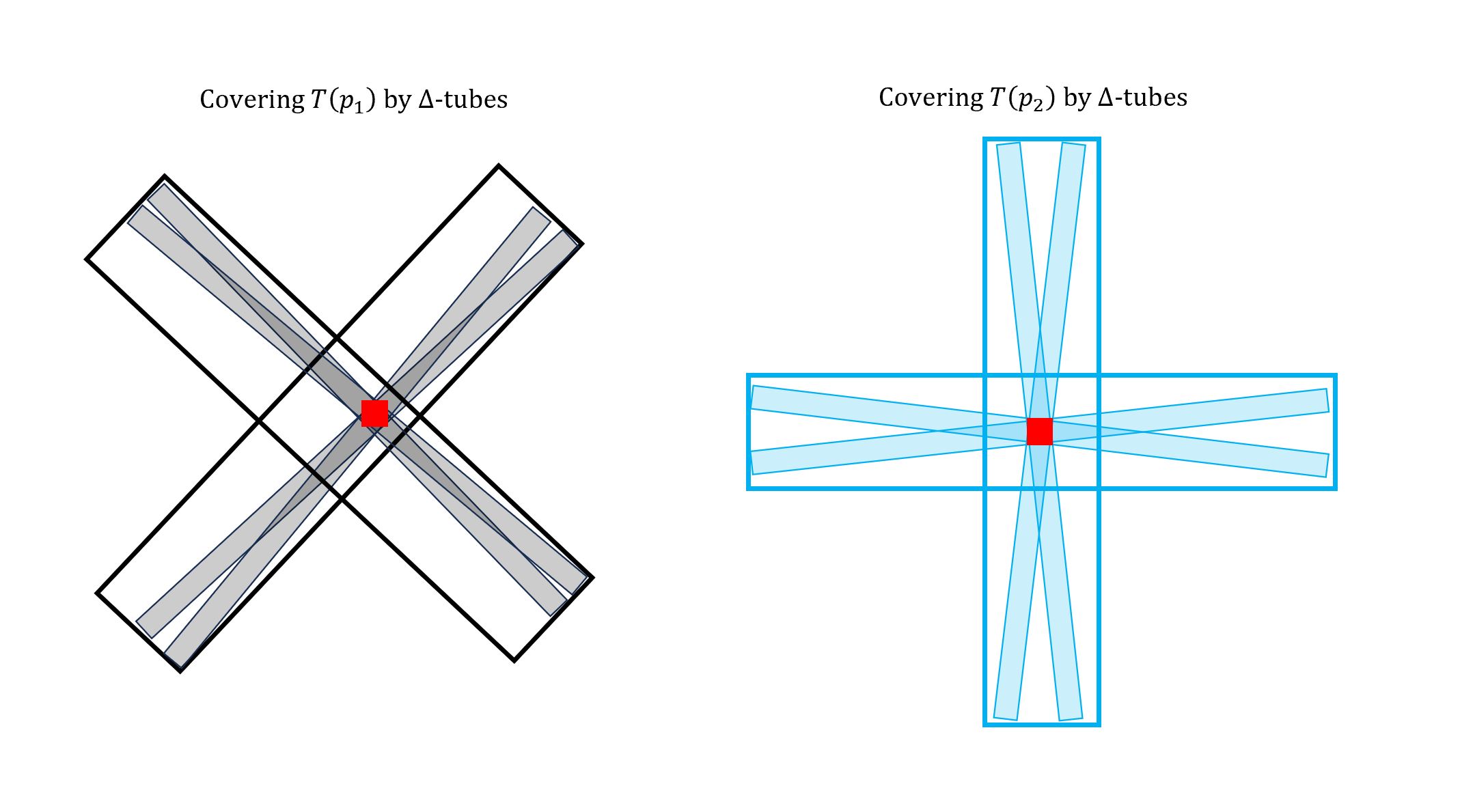}

\includegraphics[scale=.5]{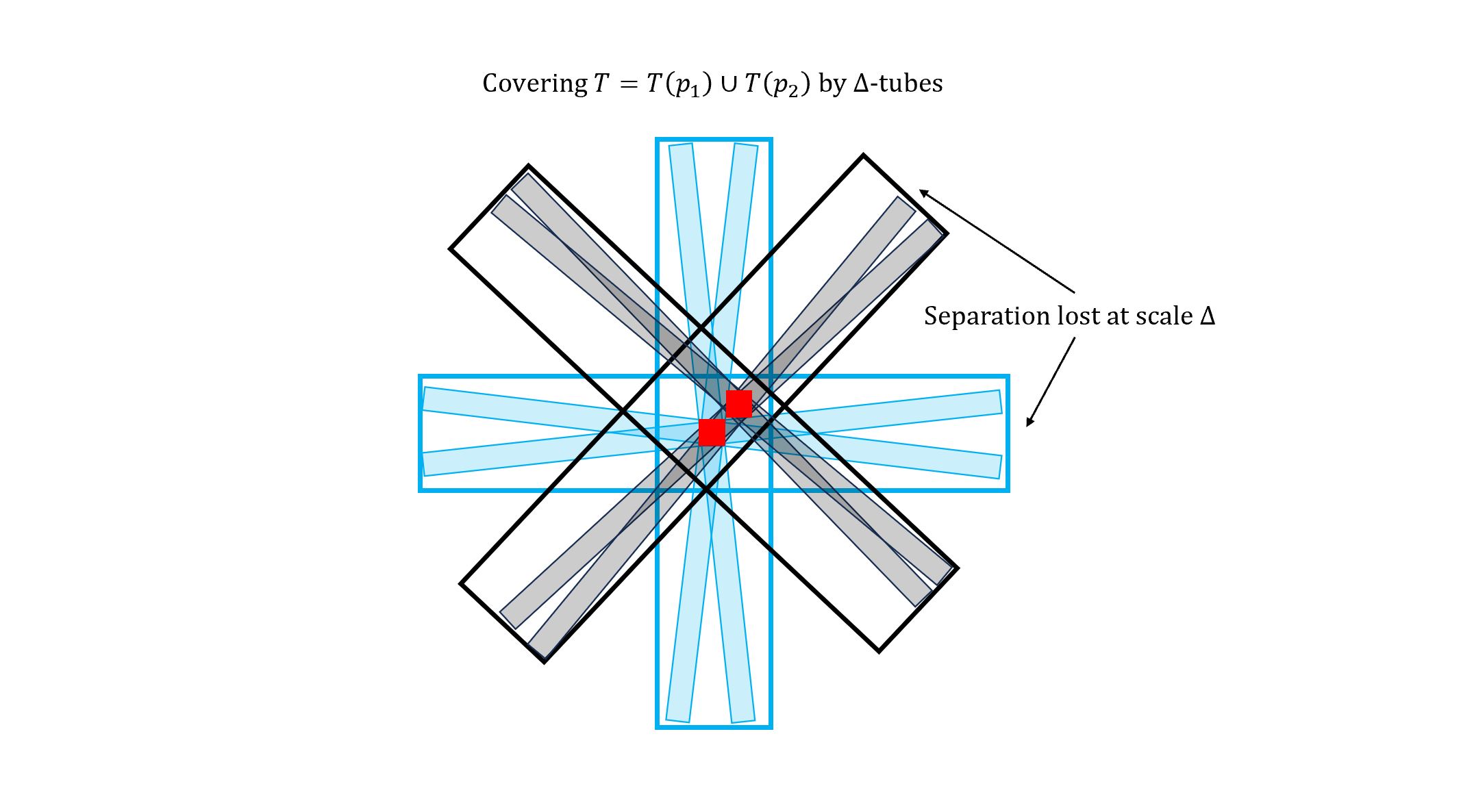}
\end{center}

So, what does each part of this lemma accomplish? Without removing too many squares, we can obtain a similarly spaced set of tubes at a larger scale (and without a much worse constant). Additionally, (2) tells us that for a fixed large tube, the number of incidences of $p$'s and small tubes within that large tube does not shrink too much with our refinement. In fact, a quick calculation reveals that $H$ is actually also an \emph{upper} bound on the \emph{average} number of incidences. This lemma will be immediately useful as we prove the main proposition of this section, so we explain the proof here. 

\begin{proof}
Suppose we start with a minimal cover of $\mathcal{T} = \bigcup_{p\in\mathcal{P}}\mathcal{T}(p)$ by a set of larger $\Delta$-tubes. We denote this cover by $\mathcal{T}_\Delta$. In order to prove the desired result, we need to increase the uniformity of our collection of squares and our collection of covering tubes, which we will do by pigeonholing a few times. 

First, we want to guarantee that our $\Delta$-tubes each contain roughly the same number of $\delta$-tubes from $\mathcal{T}(p)$. Fixing $p\in\mathcal{P}$, we can consider only the $\Delta$-tubes $\mathbf{T}\in\mathcal{T}_\Delta$ that contain between $2^{j-1}$ and $2^j$ elements of $\mathcal{T}(p)$. The largest $2^j$ we would need to consider is $M$ (the maximal size of \emph{any} $\mathcal{T}(p)$). Also, we can ignore any $2^j$'s smaller than some absolute constant times $M\Delta^2$, because there are only roughly $\Delta^2$ many $\Delta$-tubes even intersecting $[0, 1)^2$, the only tubes we may need to consider. So, some $2^j$'s are small enough that they cannot contribute much to the total number of tubes. Thus, we only have to consider the $j$'s such that $CM \Delta^2 \leq 2^j\leq M$, and there are clearly only $\approx_\Delta 1$ many of these. 

So, we can pull out a particular index $j(p)$ and consider only the $\Delta$-tubes containing approximately this many smaller tubes. This may depend on $p$, but since there are not many choices for $j$, we can pigeonhole again and remove the squares $p$ with different values for this $j$. Then, we remove the $\delta$-tubes from each $\mathcal{T}(\bar{p})$ that are not contained in the remaining $\Delta$-tubes, where $\bar{p}$ is in our refined collection. 

\begin{center}

\includegraphics[scale=.5]{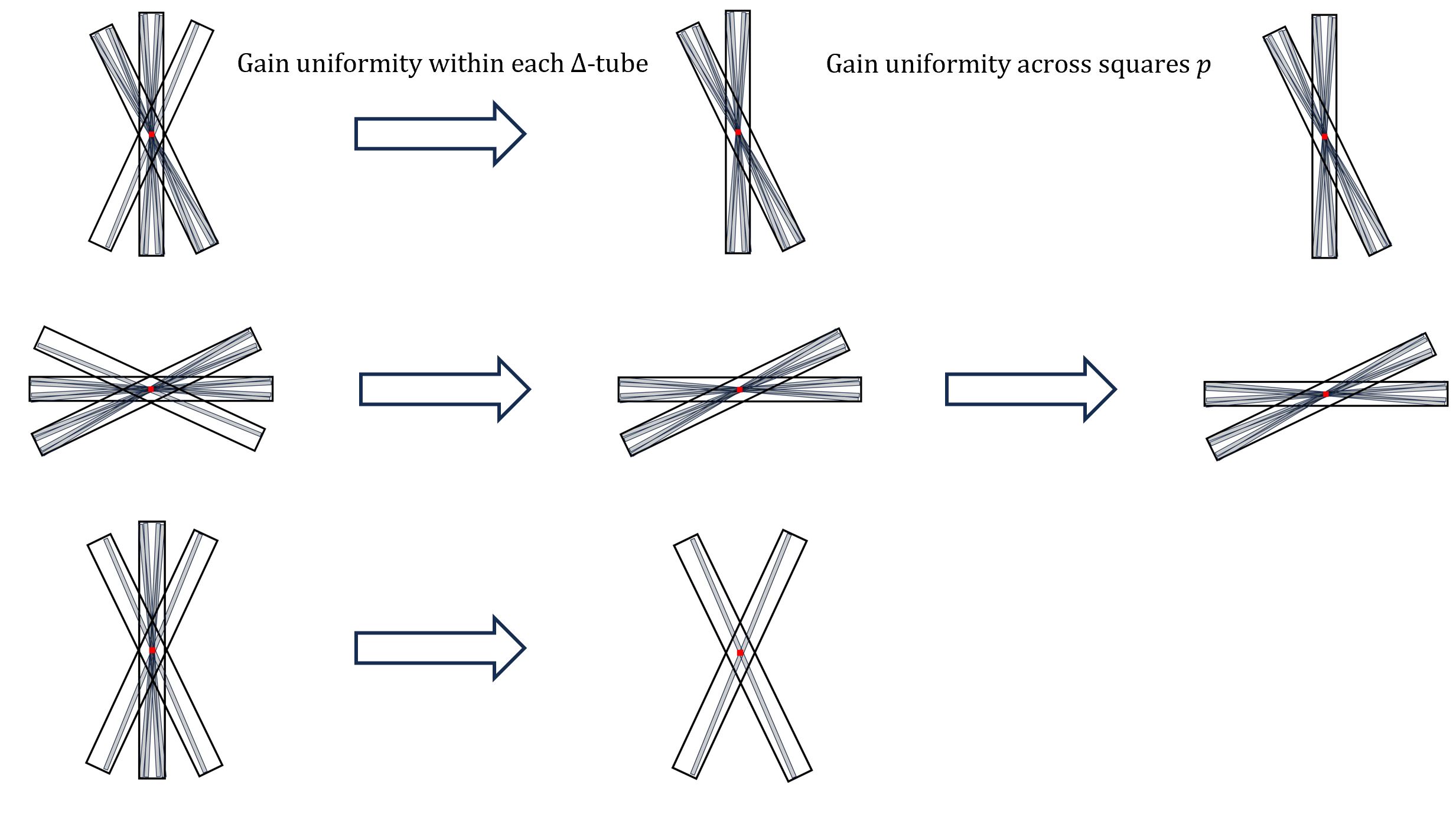}

\end{center}

The upshot is that our collections are now uniform in the sense that each refined family $\mathcal{T}(p)$ is covered by approximately the same number of $\Delta$-tubes (denoted $m_2)$, each remaining $\Delta$-tube contains approximately the same number of $\delta$-tubes (denoted $m_1)$, and the approximate equalities depend only on $\Delta$. In the above figure, $m_1$ is 4, and $m_2$ is 2. 

Unfortunately, we still need to achieve a bit more uniformity. This is accomplished with another dyadic pigeonholing, similar to the start of this proof, where we consider the $\Delta$-tubes $\mathbf{T}$ which cover the smaller $\delta$-tubes of between $2^{j-1}$ and $2^{j}$ many associated squares $p$. The remaining collection of $\Delta$-tubes is our set $\bar{\mathcal{T}}_{\Delta}$. We define $H$ to be $2^j m_1$, where $j$ is the value we extract from this pigeonholing. The requirement on the value of $H$ from Proposition \ref{prop_2scale} is satisfied since 
\begin{equation*}
    H\approx_\Delta M\cdot \vert\mathcal{P}\vert/\vert\bar{\mathcal{T}}_{\Delta}\vert
\end{equation*}
\noindent translates to
\begin{equation*}
    \text{(\# of }p\text{ represented per }\Delta\text{-tube)}\text{(}\delta\text{-tubes per }\Delta\text{-tube)} \approx_\Delta \dfrac{\text{(}\delta\text{-tubes per }p\text{)}\text{(original \# of }p\text{)}}{\text{(\# of }\Delta\text{- tubes)}}.
\end{equation*}

Now, we verify the two claims of the lemma: 
\begin{enumerate}
\item $\bar{\mathcal{T}}_{\Delta}$ is a $(\Delta, s, C_2)$ set: First, fix a scale $r\geq\Delta$. It suffices to show that a fixed $r$-tube $\mathbf{T}_r$ contains $\lessapprox_\Delta C_1 \vert \bar{\mathcal{T}}_{\Delta}\vert r^s$ $\Delta$-tubes $\mathbf{T}$ from our collection $\bar{\mathcal{T}}_{\Delta}$. From now on, we will consider only the $\Delta$-tubes in this $\mathbf{T}_r$, and we will denote by $N$ the number of these $\Delta$-tubes. A quick final pigeonholing allows us to fix a particular $p_0$ such that 
\begin{equation*}
    \vert T_\Delta(p_0) \vert \gtrapprox_\Delta N \vert\bar{T}_\Delta \vert m_2.
\end{equation*}
\noindent Here, the uniformity we gained from pigeonholing in the setup is crucial because no matter what $p_0$ we get above, we know that each $\Delta$-tube associated to $p_0$ contains about $m_1$ many $\delta$-tubes from $\mathcal{T}(p_0)$. Using the fact that $\mathcal{T}(p_0)$ is a $(\delta, s, C_1)$-set that - even after the above refinements - contains about $M$ many $\delta$-tubes, we obtain that \begin{equation*}
\vert \{T\in\mathcal{T}(p_0): T\subset \mathbf{T}_r\}\vert \lesssim C_1 M r^s.
\end{equation*}
\noindent On the other hand, from the previous paragraph, we have 
\begin{equation*}
\vert \{T\in\mathcal{T}(p_0): T\subset \mathbf{T}_r\}\vert \gtrapprox_\Delta N \vert \bar{T}_\Delta \vert^{-1} M,
\end{equation*}
\noindent and the claim follows.

\item For any $\mathbf{T}\in\bar{\mathcal{T}}_{\Delta}$,
    $\vert\{(p, T)\in\bar{\mathcal{P}}\times \mathcal{T}^\delta: T\in\mathcal{T}(P) \text{ and } T\subset\mathbf{T}\}\vert\gtrsim H$:
    This claim follows almost immediately from the definition of $H$. Fix a $\Delta$-tube $\mathbf{T}$, and we know it must contain approximately $m_1$ $\delta$-tubes from $2^j$ different $p$'s (again, where $j$ comes from the last pigeonholing). This is exactly how $H=2^jm_1$ was defined, and as a result of the pigeonholing, we know that the above holds \emph{uniformly} for the remaining $\Delta$-tubes $\mathbf{T}$. 
\end{enumerate}

\end{proof}

Equipped with this lemma, we explain the proof of the induction on scales proposition. 

\begin{proof}[Proof of Proposition \ref{prop_2scale}]
As in the proof of the lemma, we begin with some pigeonholing. Namely, we need to reduce the family $\mathcal{Q}_0$ of $\Delta$-squares intersecting our starting family $\mathcal{P}_0$. First, fix some $Q\in\mathcal{Q}$ and apply Lemma \ref{Pigeonholing before induction} to $\mathcal{P}\cap Q$, obtaining a set $\mathcal{T}_\Delta(Q)$ of $\Delta$-tubes. The point of this application is that we get a uniform lower bound on the cardinality of a set of incidences, and we retain that our resulting collection of tubes is ``$s$-dimensional" at our new scale, which was a requirement of the ``nice" configurations. 

Having already seen some pigeonholing in the previous proof, we elide the details of the several uses of this principle and focus on explaining the gain in uniformity. The interested reader is encouraged to look at the (very thorough) Section 5 of the original paper \cite{OS21} for specifics. First, we guarantee that the families $\mathcal{T}_\Delta(Q)$ have roughly constant cardinality, necessitating the removal of some squares. Then, we force the number of $\delta$-tubes from $\mathcal{T}_0$ in each $\mathbf{T}$ to be roughly constant, which may require the removal of more $Q$. We denote the remaining set of $\Delta$-squares by $\mathcal{Q} $.\footnote{Technically, to ensure we match the exact definition of niceness, we still may need to throw out a few $\Delta$-tubes associated with each of the remaining $Q\in\mathcal{Q}$ so their cardinalities agree exactly, but this will not be a problem since if anything this \emph{increases} the separation of the remaining tubes.}

Claims (1) and (2) follow because the loss in each portion of the pigeonholing depended only on $\delta$, and (3) follows from the fact that we applied the previous lemma to each square $Q$ individually and could only gain separation of the $\Delta$-tubes associated to each as we progressed. So, (4) and the associated inequality are what remains to be shown. This inequality follows from a few elementary observations, and
\begin{equation*}
    \vert\mathcal{T}(Q)\vert\gtrapprox_\delta\dfrac{\vert\mathcal{T}_Q\vert}{M_Q}M.
\end{equation*}
So we turn to construct the families $T_Q$, which satisfy the above and property (4). 

Write $\bar{\delta}:=\delta/\Delta$, which is the scale we are interested in. The main problem is that we have $\delta$-separation but not $\bar{\delta}$-separation that property (4) requires. To fix this, we consider tube packets, which are the tubes from a $\mathcal{T}(p)$ which lie within some $\bar{\delta}$-tube $\mathbf{T}_{\bar{\delta}}$.  We perform a final pigeonholing on each family $\mathcal{P}\cap Q$ to ensure that for all $p\in\mathcal{P}$, the associated family of $\delta$-tubes has cardinality roughly $M_Q$.

The main idea now is essentially to select one tube from each tube packet, as this will upgrade the $\delta$-separation to $\bar{\delta}$-separation. The families $T_Q$ we want are (more or less) the images of these single remaining tubes from each packet under the map $S_Q$. Technically, we may not be able to select exactly one tube from each tube packet prior to defining the above families. If we were, the desired inequality would follow since the tube packets represented by the tubes would be disjoint, and we could add the cardinalities of each tube packet. Fortunately, we can further refine our collection of tubes so that this holds, and we do not sacrifice many tubes. This follows from two very brief lemmas, that tubes which are ``separated" (in a certain sense) represent different tube packets, and that we only have to sacrifice a small number of tubes to obtain a separated set. Applying these lemmas establishes the desired inequality and completes the proof. 

\end{proof}

\section{Multi-scale decomposition} \label{Section 5}

In this section, we will prove the existence of a sequence of scales where we obtain $\delta^{-\epsilon}$ gain on the number of tubes $|\mathcal{T}|$. Then, we will prove Theorem \ref{MainT_tube}.

\begin{definition}
    Let $\delta \in 2^{-\mathbb{N}}$ be a dyadic number such that also $\delta^{1/2} \in 2^{-\mathbb{N}}$. Let $C,K >0$, and let $0 \leq s \leq d$. A non-empty set $\cp \subset D_\delta$ is called $(\delta, s, C,K)$-regular if $\cp$ is a $(\delta,s,C)$-set, and moreover
    \begin{equation*}
        |\cp|_{\delta^{1/2}} \leq K \cdot \delta^{-s/2}
    \end{equation*}
    where $|\cp|_{\delta^{1/2}}$ is the smallest number of $\delta^{1/2}$-cubes in $\cd_{\delta^{1/2}}$ that we need to cover all $\delta$-cubes in $\cp$.
\end{definition}

\begin{remark}
\rm

The definition of regular sets is weaker than the AD-regular set: it only requires regularity at one scale, $\de^{1/2}$.

\end{remark}

\begin{remark}\rm
    We consider $(\delta,s,C)$-sets $\cp$ such that $|\cp|_\delta \approx \delta^{-s}$. More specifically, $C^{-1} \delta^{-s} \leq |\cp|_\delta \leq \delta^{-s}$. Thus, we have $|\cp \cap Q|_{\delta} \leq C\delta^{-s/2} $ where $Q$ is a $\delta^{1/2}$-cube. If $\mathcal{P}$ is $(\delta,s,C,K)$-regular, it means 
    \begin{equation*}
        \frac{1}{|\cp|_{\delta^{1/2}}} \sum_{ Q \in \cd_{\delta^{1/2}}(\cp)} |\cp \cap Q| = \frac{|\cp|_{\delta}}{|\cp|_{\delta^{1/2}}} \geq \frac{1}{K}\delta^{-s/2}. 
    \end{equation*}
    Therefore, $|\cp|_{\delta} \approx \delta^{-s}$ and $|\cp|_{\delta^{1/2}} \approx \delta^{-s/2}$.
\end{remark}
\begin{theorem}\label{thm_card_regT}
    Given $s \in (0,1)$ and $ t \in (s,2]$, there exists $\epsilon= \epsilon(s,t) >0$ such that the following holds for small enough $\delta \in 2^{-\mathbb{N}}$. Let $\cp \subset \cd_\delta$ be a $(\delta, t, \delta^{-\epsilon}, \delta^{-\epsilon})$-regular set. Assume that for every $p \in \cp$, there exists a $(\delta, s, \delta^{-\epsilon})$-set of dyadic tubes $\cT(p) \subset \cT^\delta$ such that $T \cap p \neq \emptyset$ for all $ T \in \cT(p)$. Then,
    \begin{equation*}
        |\cT| \geq \delta^{-2s-\epsilon},
    \end{equation*}
    where $\cT = \cup_{p \in \cp} \cT(p)$.
\end{theorem}
    We will use the result that under the assumptions above, we have either 
    \begin{equation}\label{T_or_result}
        |\cT| \gtrsim \delta^{-2s-\epsilon} \qquad \text{or} \qquad |\cT_{\delta^{1/2}}| \geq \delta^{-s-\epsilon}
    \end{equation}
    where $\cT_{\delta^{1/2}} := \cT^{\delta^{1/2}}(\cT)$ is the smallest set of $\delta^{1/2}$-tubes that we need to cover all tubes in $\cT$.
    
    The proof of \eqref{T_or_result} will be provided in Section \ref{est_w_reg}. If the first happens, we are done. So, we will assume the second in the proof of Theorem \ref{thm_card_regT}.
\begin{proof}
    Without loss of generality, we assume that $\cp \approx \delta^{-t}$ and $|\cT(p)| =M \approx \delta^{-s}$ for all $p \in \cp_0$ and it suffices to consider \textit{heavy} squares $Q \in \cd_{\delta^{1/2}}(\cp)$, which means $|\{ p \in \cp : p \subset Q\}| \geq \delta^{-t/2 +5\epsilon}$. By Proposition \ref{prop_2scale} (3), we can additionally assume that 
    \begin{equation*}
        (\cQ, \cT_{\delta^{1/2}}):=(\mathcal{D}_{\delta^{1/2}}(\cp), \cT^{\delta^{1/2}}(\cT))
    \end{equation*}
    is a $(\delta^{1/2}, s, \mathbf{C}, \mathbf{M})$-nice configuration where $\mathbf{C}:=C_{\delta^{1/2}} \approx 1$ and $\mathbf{M}:=M_{\delta^{1/2}} \gtrapprox \delta^{-s/2}$. Thus, using Corollary \ref{cor_T_est}, we obtain 
    \begin{equation}\label{T_deltahalf_1}
        |\cT_{\delta^{1/2}}| \gtrsim \mathbf{M}\delta^{-s/2} \cdot (\mathbf{M} \delta^{s/2})^{\frac{t-s}{1-s}}.
    \end{equation}
    Since, we assumed the second case of \eqref{T_or_result}, we have $|\cT_{\delta^{1/2}}| \geq \delta^{-s-C\epsilon}$ for some $C>0$ and sufficiently small $\epsilon >0$. If $\mathbf{M} \geq \delta^{-s/2-C\epsilon/2}$, we use \eqref{T_deltahalf_1}. Otherwise, we use $|\cT_{\delta^{1/2}}| \geq \delta^{-s-C\epsilon}$. In either case, we get
    \begin{equation}\label{T_deltahalf_2}
        |\cT_{\delta^{1/2}}| \geq \mathbf{M} \delta^{-s/2 -C_1 \epsilon}
    \end{equation}
    for some $C_1 \gtrsim_{s,t} C$. Since $Q$ are heavy squares, it follows from Proposition \ref{prop_2scale} (1) and Corollary \ref{cor_T_est} that 
    \begin{equation}\label{T_Q}
        |\cT_Q| \gtrapprox M_Q \delta^{-s/2}
    \end{equation}
    where $M_Q$ is the one in Proposition \ref{prop_2scale} (4). Combining \eqref{prop_2scale_eq}, \eqref{T_deltahalf_2}, \eqref{T_Q}, we get
    \begin{equation*}
        |\cT| \gtrapprox \frac{|\cT_{\delta^{1/2}}|}{\mathbf{M}} \frac{|\cT_Q|}{M_Q} M \gtrapprox M \delta^{-s-C_1\epsilon} \gtrsim \delta^{-2s-C_1\epsilon}.
    \end{equation*}
\end{proof}
\subsection{Multi-scale decomposition}

\begin{definition}
\label{part3_key-def}
Let $0<\de<\De\leq 1$ be dyadic numbers and let $\cp\subset[0,1)^2$. For given $0\leq s\leq 2$ and $C>0$:
\begin{enumerate}
    \item We say that $\cp$ is an $(s, C)$-set between the scales $\de$ and $\De$ if $S_Q(\cp\cap Q)\subset [0,1)^2$ is a $(\de/\De,s,C)$-set for all $Q\in\cd_\De(\cp)$.
    \item We say that $\cp$ is $(s,C,K)$-regular between the scales $\de$ and $\De$ if $S_Q(\cp \cap Q)$ is a $(\de/\De,s,C)$-set with an additional property
    \begin{equation*}
        |S_Q(\cp\cap Q)|_{(\de/\De)^{1/2}}\leq K(\de/\De)^{-s/2}
    \end{equation*}
    for all $Q\in\cd_\De(\cp)$. 
\end{enumerate}

\end{definition}

Given scales $\de=\De_n <\De_{n-1}<\cdots<\De_1\leq \De_0=1$, by a bottom-to-top pigeonholing (with an implicit loss depending on $n$), we can assume the point set $\cp$ is $\De_j$-uniform in the sense that
\begin{equation*}
    |\cp \cap Q|_{\De_j}=N_j
\end{equation*}
for all $j=1,\ldots, n$ and all $Q\in\cd_{\De_{j-1}}(\cp)$. 

Let $\Delta \in 2^{-\mathbb{N}}$ and $\delta = \Delta^{m}$ for some $m \in \mathbb{N}$. The following proposition gives a multi-scale decomposition of a $(\Delta_i)_{i=1}^m$-uniform $(\delta, t, \delta^{-\epsilon})$-set $\cp$.
\begin{proposition}\label{prop_scales}
Given $s\in(0,1)$, $t\in(s,2]$, $\De\in2^{-\ZN}$ and $\e>0$ there is $0<\tau=\tau(\e,s,t)\leq\e$ such that the following holds for large enough $m$.

Let $\de=\De^m$ and let $\cp\subset[0,1]^2$ be a $(\De^i)_{i=1}^m$-uniform $(\de,t,\de^{-\e})$-set. Then there are numbers $t_j\in[s,2]$, $1\leq j\leq n$, and scales
\begin{equation*}
    \de=\De_n<\De_{n-1}<\cdots<\De_1<\De_0=1,
\end{equation*}
with $\De_j$ an integer power of $\De$, and a partition $\{1,\ldots, n\}=\cs\bigcup\cb$ (structured and bad indices) such that the following properties hold:
\begin{enumerate}
    \item $\De_{j-1}/\De_j\geq\de^{-\tau}$ for all $j\in\cs$ and $\prod_{j\in\cb}(\De_{j-1}/\De_j)\leq \de^{-\e}$.
    \item For each $j\in\cs$, the set $\cp$ is a $(t_j,(\De_{j-1}/\De_j)^\e)$-set between the scales $\De_j$ and $\De_{j-1}$. Moreover, if $t_j>s$ then $\cp$ is $(t_j,(\De_{j-1}/\De_j)^\e, (\De_{j-1}/\De_j)^\e)$-regular between the scales $\De_j$ and $\De_{j-1}$. 
    \item $\prod_{j\in\cs}(\De_{j-1}/\De_j)^{t_j}\geq\de^{\e-t}$.
    \item If $j\in\cb$, then $j+1\not\in\cb$ for all $j\in\{1,\ldots, n-1\}$.
\end{enumerate}
\end{proposition}
\begin{proof}
Let us explain the idea of the proof of Proposition \ref{prop_scales}. We define a function $f=f_P:[0,m] \rightarrow [0,2m]$ by letting $f(0)=0$ and
\begin{equation*}
    f(j) = \sum_{i=1}^j \frac{\log(N_i)}{\log(1/\De)}, \qquad 1 \leq j \leq m
\end{equation*}
and interpolating linearly. By using $f$, we consider exponents such that $N_1 \cdots N_j = \Delta^{-f(j)}$. 

Now, we introduce the following definition. Given a function $f:[a,b] \rightarrow \mathbb{R}$, we let
    \begin{equation*}
        L_{f,a,b}(x):=f(a) +s_f(a,b)(x-a), \qquad x \in \mathbb{R}
    \end{equation*}
    and $s_{f}(a,b)$ be the slope of $L_{f,a,b}$. We say that $(f,a,b)$ is \textit{$\epsilon$-linear} if 
    \begin{equation*}
        |f(x)-L_{f,a,b}(x) | \leq \epsilon |b-a| , \qquad x \in [a,b]
    \end{equation*}
    and we say that $(f,a,b)$ is \textit{$\epsilon$-superlinear} if
    \begin{equation*}
        f(x) \geq L_{f,a,b}(x) -\epsilon|b-a|, \qquad x \in [a,b].
    \end{equation*}
    
Note that $s\in (0,1)$ and $t \in (s,2]$ and $f$ is a $2$-Lipschitz function. Then, there exists $\tau >0$ and a family of non-overlapping intervals $\{[c_j,d_j]\}_{j=1}^n$ contained in $[0,m]$ such that:
\begin{enumerate}
    \item For each $j$, at least one of the following alternatives holds:\\
    (a) $(f,c_j,d_j)$ is $\epsilon$-linear with $s_f(c_j,d_j) \geq s$.\\
    (b) $(f,c_j,d_j)$ is $\epsilon$-superlinear with $s_f(c_j,d_j)=s$.
    \item $d_j-c_j \geq \tau m$ for all $j$.
    \item $|[0,m] \backslash \cup_j [c_j,d_j] |\lesssim_{s,t} \epsilon m$.
\end{enumerate}
For the detail of the proof, see Lemma 8.5 in \cite{OS21}. If $(f,c_j,d_j)$ is $\epsilon$-superlinear, $P$ is a $( s_f(c_j,d_j), (\Delta_{j-1} / \Delta_j)^\epsilon)$-set between scales $\Delta_j := \Delta^{d_j}$ and $\Delta_{j-1}:=\Delta^{c_j}$. Similarly, if $(f,c_j,d_j)$ is $\epsilon$-linear, $P$ is a $( s_f(c_j,d_j), (\Delta_{j-1} / \Delta_j)^\epsilon,(\Delta_{j-1} / \Delta_j)^\epsilon)$-regular between scales $\Delta_j := \Delta^{d_j}$ and $\Delta_{j-1}:=\Delta^{c_j}$. The property (3) means $|[0,m] \backslash \cup_j [c_j,d_j] |$ is negligible, which implies $\prod_{j \in \cB} (\Delta_{j-1} / \Delta_j) \leq \delta^{-\epsilon}$ in (1) of Proposition \ref{prop_scales}.
    
\end{proof}

Let us state the proposition where we obtain the lower bound of $|\mathcal{T}|$ according to the properties of scales.

\begin{proposition}\label{propT_multiscale}
    Given $s \in (0,1) , t>s, \tau \in (0,1), n \geq 1$, if $\epsilon_G , \eta,\lambda >0 $ are taken small enough and $0 < \epsilon_N \leq \epsilon_G$, then the following holds for all $C_p \geq 1$ and $0 < \delta \leq 1$. 

    Let $(\mathcal{P},\mathcal{T}) \subset D_\delta \times \cT^\delta$ be a $(\delta, s, \delta^{-\lambda}, M)$-nice configuration for some $M \geq 1$. Let 
    \begin{equation*}
        \delta = \Delta_n < \Delta_{n-1} < \cdots < \Delta_1 < \Delta_0=1
    \end{equation*}
    be a sequence of dyadic scales, and assume that $\mathcal{P}$ is $(\Delta_j)_{j=1}^n$-uniform. we assume that the scale indices $\{1,\cdots, n\}$ are partitioned into {normal scales, good scales}, and {bad scales}, denoted $\mathcal{N}, \mathcal{G}$ and $\mathcal{B}$, respectively. We assume that 
    \begin{equation*}
        \Delta_j / \Delta_{j-1} \leq \delta^\tau, \qquad j \in \mathcal{N} \cup \mathcal{G}.
    \end{equation*}
    Moreover, the family $\mathcal{P}$ has the following structure at the normal and good scales:
    \begin{itemize}
        \item If $j \in \cn$, then $\mathcal{P}$ is an $(s, [\log(1/\delta)]^{C_p}\cdot (\Delta_{j-1}/\Delta_j)^{\epsilon_N})$-set between the scales $\Delta_j$ and $\Delta_{j-1}$.
        \item If $j \in \cg$, there exists a number $t_j \geq t$ such that $\mathcal{P}$ is
        \begin{equation*}
            (t_j, [\log(1/\delta)]^{C_p}\cdot(\Delta_{j-1}/\Delta_j)^{\epsilon_G}, [\log(1/\delta)]^{C_p}\cdot (\Delta_{j-1}/\Delta_j)^{\epsilon_G})\textit{-regular}
        \end{equation*}
        between scales $\Delta_j$ and $\Delta_{j-1}$.
        \item Otherwise, $j \in \cb$.
    \end{itemize}
    Then, there exists constants $C,C'>0$ such that
    \begin{equation}\label{T_multiscale}
        |\cT| \geq \left[\log\left(\frac{1}{\delta}\right)\right]^{-C} M \delta^{C'\lambda} \delta^{-s+\epsilon_N} \prod_{j \in \cg} \left(\frac{\Delta_{j-1}}{\Delta_j} \right)^\eta \cdot \prod_{j \in \cb} \frac{\Delta_j}{\Delta_{j-1}}.
    \end{equation}
\end{proposition}
    There is no gain in normal scales $\mathcal{N}$, but there is no loss as well. The gain comes from good scales $\mathcal{G}$, and the loss comes from bad scales $\mathcal{B}$. In the proof of Theorem \ref{MainT_tube}, we combine Proposition \ref{prop_scales} and \ref{propT_multiscale} to construct a sequence of $\Delta_n$ such that there is $\delta^{-\epsilon}$ gain from good scales $\mathcal{G}$ while the loss from bad scales $\mathcal{B}$ is negligible compared to the gain.
\begin{proof}
    Since the proof uses induction, we consider only when $n=1$. Then, there are only two scales, $\delta$ and $1$. If $1 \in \cb$, \eqref{T_multiscale} follows from the assumption that $|\cT| \geq M$. If $1 \in \cn$, we use Corollary \ref{cor_T_est} and obtain that $|\cT| \gtrapprox_\delta M\cdot \delta^{\lambda} \cdot \delta^{-s+\epsilon_N} $. Lastly, if $1 \in \cg$, since $t_j \geq s$, we use Corollary \ref{cor_T_est} again and obtain that
    \begin{equation*}
        |\cT| \gtrapprox_\delta \delta^{\epsilon_G} \cdot M\delta^{-s} \cdot (M\delta^s)^{\frac{t-s}{1-s}}.
    \end{equation*}
    Let 
    \begin{equation*}
        \gamma = \frac{(\eta+2\epsilon_G) (1-s)}{(t-s)}.
    \end{equation*}
    If $M \geq \delta^{-s-\gamma}$, \eqref{T_multiscale} immediately follows and if $M\leq \delta^{-s-\gamma}$, it follows from Theorem \ref{thm_card_regT}, where the regularity of $\cp$ is used.
\end{proof}

Now, we can prove the main theorem of the paper.

\begin{proof}[Proof of Theorem \ref{MainT_tube}]
    Without loss of generality, we assume that $\delta= \Delta^m$ where $\Delta \in 2^{-\mathbb{N}}$, $\cT(p)$ are $(\delta, s, \delta^{-\lambda})$-sets such that $|\cT(p)|=M \geq \delta^{-s+\lambda}$. We can also assume that $\cp$ is a $(\Delta^i)_{i=1}^m$-uniform $(\delta, t, \delta^{-\epsilon})$-set. We choose $\epsilon_G $ sufficiently small and choose $\epsilon_N, \epsilon, \lambda$ such that $\epsilon_N=\epsilon \leq \epsilon_G \eta/100$ and $\lambda \leq \epsilon_G \eta/(100C')$ where $\eta$ and $C'$ constants from Proposition \ref{propT_multiscale}. Combining Proposition \ref{prop_scales} and \ref{propT_multiscale}, we obtain the desired estimate for $|\cT|$.
    To be more specific, Proposition \ref{prop_scales} implies that the loss from bad scales is negligible and the gain from good scales is at least $\delta^{-\epsilon_G \eta/8}$.
\end{proof}

\section{Improved estimates under regular conditions}\label{est_w_reg}
In this section, we prove \eqref{T_or_result}. Recall the definition of a $(\de,s,C)$-set in Definition \ref{def-1}. It gives that $|\cT(p)|_\de\gtrapprox\de^{-s}$ and $|\cT(p)|_{\de^{1/2}}\gtrapprox\de^{-s/2}$ for any $p\in \cp$.  Suppose, on the contrary, we have 
\begin{equation}
\label{negation-1}
    |\cT|_\de\lessapprox\de^{-2s} \text{ and }|\cT|_{\de^{1/2}}\lessapprox\de^{-s}.
\end{equation}
By Proposition \ref{discretized-ST}, we know in addition 
\begin{equation}
\label{negation-2}
     \de^{-2s}\lessapprox|\cT|_\de,\,\,\,\de^{-s}\lessapprox|\cT|_{\de^{1/2}},\,\,\,|\cT(p)|_\de\lessapprox\de^{-s},\text{ and }|\cT(p)|_{\de^{1/2}}\lessapprox\de^{-s/2}.
\end{equation}
The above estimates gives that any $\de^{1/2}$-tube intersecting $p$ contains $\approx\de^{-s/2}$ $\de$-tubes in $\cT(p)$. From there, we will construct a configuration (see Figure \ref{figure2}) that violates Bourgain's projection theorem. This shows that \eqref{negation-1} cannot be true.

\smallskip

We begin by considering a typical $\de^{1/2}$-cube $Q$ that $Q\cap \cp\neq \varnothing$. $Q\cap \cp$ can be understood as a $(\de^{1/2}, t, C)$-set (after rescaling with factor $\de^{-1/2}$), so $|Q\cap \cp|_\de\gtrapprox \de^{-t/2}$. Note that each $\de^{1/2}$-tube intersecting $Q$ corresponds to a direction (with resolution $\de^{1/2}$), and the set of these directions, denoted by $\Si=\Si(Q)$, is a $(\de^{1/2}, s, C)$-set. Since $t>s$, we claim that for most directions $\si\in\Si$, $\pi_\si(\cp\cap Q)$ contains a $(\de^{1/2}, s, C)$-set (after a $\de^{-1/2}$-dilation). This is done by using Kaufman's exception set estimate in the original \cite{OS21}. We give a combinatorial sketch below.  

\subsection{\texorpdfstring{$\pi_\si(\cp\cap Q)$}{pi sigma(P cap Q)} contains a \texorpdfstring{$(\de^{1/2}, s, C)$}{(delta {1/2}, s, C)}-set for most directions \texorpdfstring{$\si\in\Si(Q)$}{sigma in Sigma(Q)}}

We first show that for $\gtrapprox\de^{-s}$ directions $\si\in\Si$, $|\pi_\si(\cp\cap Q)|_\de\gtrapprox\de^{-s/2}$. Let $\cp'=\de^{-1/2}(\cp\cap Q)$ be a dilation of $P\cap Q$ with factor $\de^{-1/2}$, and let $\De=\de^{1/2}$. Assume without loss of generality, $s=t$ and $|\cp'|_{\De}\approx\De^{-s}$ (or we can pick a subset to fulfill our requirements). For any distance $\De\leq\rho\leq1$, the fact that $\cp'$ is a $(\De, s, C)$-set gives that there are at most $\De^{-s}(\De/\rho)^{-s}$ pairs of $\De$-balls in $B_1,B_2\subset N_\De(\cp')$ so that $\dist(B_1,B_2)\sim\rho$. Since $\Si$ is a $(\De, s, C)$-set, each these pair $(B_1,B_2)$ determines at most $\rho^{-s}$ directions in $\Si$, and hence each pair $(B_1,B_2)$ intersects at most $\rho^{-s}$ $\De$-tubes whose direction is contained in $\Si$. Sum up all dyadic $\rho$ so
\begin{equation*}
    \sum_{B_1,B_2,T}\Id_T(B_1)\Id_T(B_2)\lessapprox\De^{-2s}.
\end{equation*}
A double-counting argument gives  for $\gtrapprox\de^{-s}$ directions $\si\in\Si$, $|\pi_\si(\cp\cap Q)|_\de\gtrapprox\De^{-s}=\de^{-s/2}$. 

A similar argument should give that for $\gtrapprox\de^{-s}$ directions $\si\in\Si$, $\pi_\si(\cp\cap Q)$ contains a $(\de^{1/2}, s, C)$-set (after a $\de^{-1/2}$-dilation). For any typical $Q$, remove those bad directions in $\Si=\Si(Q)$ so that for any $\si\in\Si(Q)$, $\pi_\si(\cp\cap Q)$ contains a $(\de^{1/2}, s, C)$-set (after a $\de^{-1/2}$-dilation). Note that after removing all bad directions (Equiv. $\de^{1/2}$-tube), $\cT(p)$ is still a $(\de, s, C)$-set, \eqref{negation-1} and \eqref{negation-2} continue to hold.

\subsection{Inside a typical \texorpdfstring{$\de^{1/2}$}{delta {1/2}}-tube}

By pigeonholing, there is a fraction $\gtrapprox1$ of $\de^{1/2}$-tube, each of which contains $\approx \de^{-t+s}$ typical $\de^{1/2}$-cubes. We claim that in addition there is another fraction $\gtrapprox1$ of $\de^{1/2}$-tube so that the $\approx \de^{-t+s}$ typical $\de^{1/2}$-cubes contained in each of these $\de^{1/2}$-tube are distributed like a $(\de^{1/2}, t-s, C')$-set with some $C'\approx 1$.

In fact, observe that for each of the $O(\log\De)$ scales $\rho\in\{1, |\log\De|, \ldots, |\log\De|^n=\De\}$, the number of $\De$-tubes $\bar T$ (which we called bad $\De$-tubes) containing $\gtrapprox\rho^{s-t}|\log\De|^{10}$ typical $\De$-cubes in $\bar T\cap B_\rho$ for some $\rho$-ball $B_\rho$ (i.e. these $\De$-cubes are trapped in a $\rho\times\De$-subtube of the original $\De$-tube $\bar T$) is $\lessapprox|\log\De|^{-1}$. This follows from a similar double counting argument as above (by counting triples $(Q_1,Q_2,T)$ so that $Q_1,Q_2\subset T$ and $\dist(Q_1,Q_2)\sim\rho'$ for some $\De\leq\rho'\leq\rho$. Here we use the fact that $\cp$ is regular, which gives the number of typical $\De$-balls is $\lessapprox\De^{-t}$). For each scale $\rho$, remove all those bad $\De$-tubes. The remaining $\De$-tubes thus satisfy what we claimed above, and the remaining $\De$-tubes is a fraction $\gtrapprox1$ of the original $\De$-tubes. 
\subsection{Obtaining a special configuration from one \texorpdfstring{$\de^{1/2}$}{delta {1/2}}-tube}

Now pick one of the remaining $\De$-tube $\bar T$ by pigeonholing, and let $\pi$ be the projection along the direction of $T$ (assume this is the vertical direction). For each typical $\De$-cube $Q\subset T$, viewed $\pi(\cp\cap Q)$ as a one-dimensional set contained in the lower horizontal side of $Q$. Hence for each $Q\subset \bar T$, $\pi(\cp\cap Q)$ contains a $(\De,s,C)$-set (after a horizontal rescaling with the factor $\De^{-1}$), and vertically, all these one-dimensional $(\De,s,C)$-sets are distributed like a $(\De,t-s,C')$-set (since vertically all the typical $\De$-ball with $Q\subset \bar T$ are distributed like a $(\De,t-s,C')$-set). Denote by $P'=\bigcup_Q\pi(\cp\cap Q)$, so for each $p\in P'$, there are $\approx\de^{-s/2}$ $\de$-tubes (denoted by $\cT'(p)$) intersecting $p$, and all these $\de$-tubes are contained in $\bar T$, and are distributed like a $(\De,s,C')$-set after recaling with a factor $\De^{-1}$. Moreover, $|\bigcup_{p\in P'}\cT'(p)|\lessapprox\De^{-2s}$ (which is a consequence of the pigeonholing at the beginning of this paragraph. Note that we always have $|\bigcup_{p\in P'}\cT'(p)|\gtrapprox\De^{-2s}$).\\

Rescale $P'$ horizontally by the factor $\De^{-1}$ and denote the resulting point set by $P''$. We obtained the following structure:

\begin{enumerate}
    \item $P''$ is obeys a product structure: $P''\subset X\times Y$ where $X$ is a horizontal line, $Y$ is a vertical $(\De,t-s,C')$-set, and for $y\in Y$, $X\times \{y\}\cap P''$ is a horizontal $(\De,s,C)$-set.
    \item For each $p\in P''$, there is a set $\bar\cT(p)$ of $\De$-tubes intersecting $p$, with $|\bar\cT(p)|\approx\De^{-s}$.
    \item $|\bar\cT:=\bigcup_{p\in P''}\bar\cT(p)|\lessapprox\De^{-2s}$.
\end{enumerate}
\begin{figure}[ht]
\centering
\includegraphics[width=0.3\textwidth]{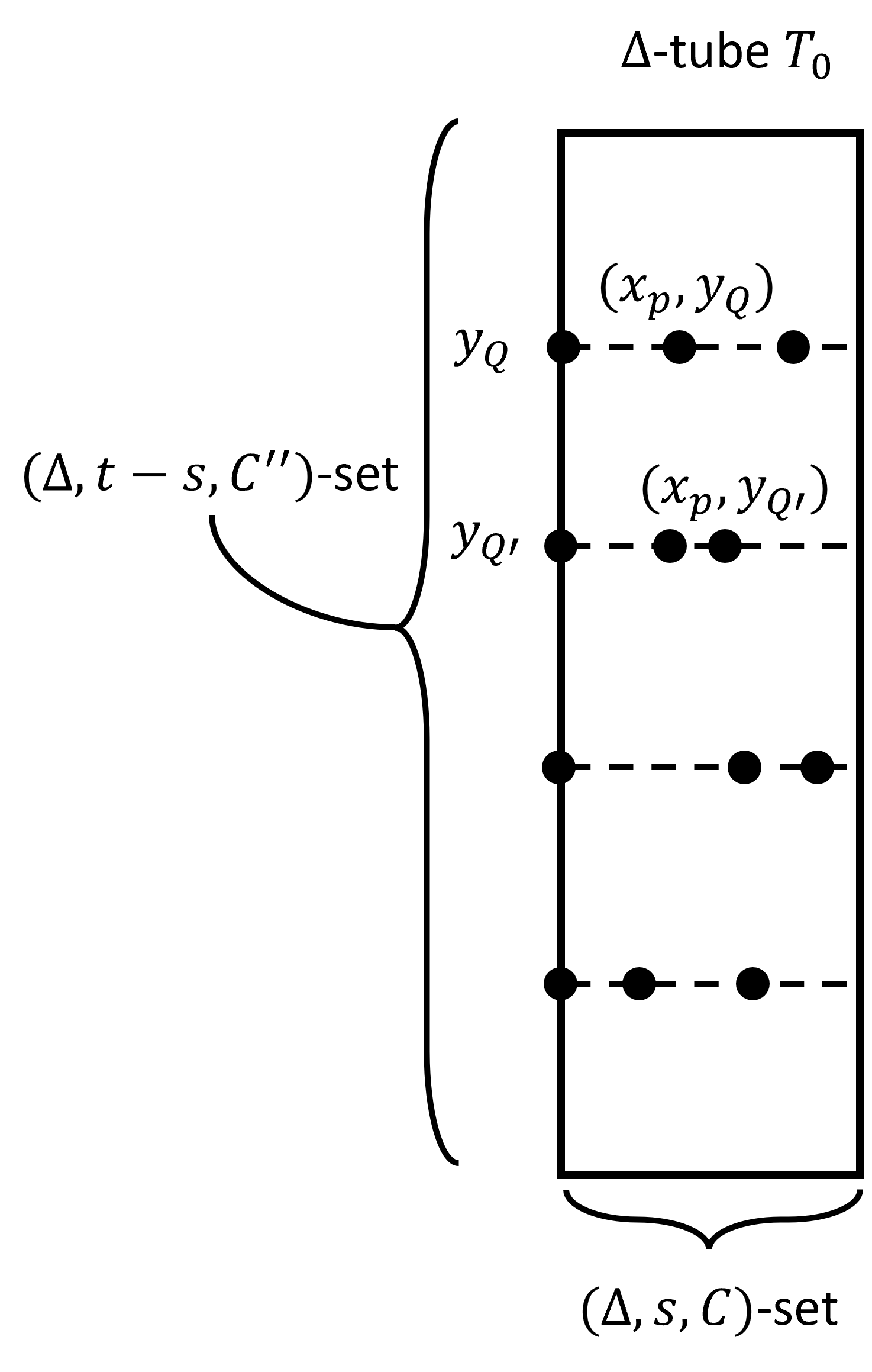}
\caption{Product structure}
\label{figure2}
\end{figure}

\subsection{A contradiction to Bourgain's projection theorem}

We want to get a contradiction between the configuration $P''$ and Bourgain's projection theorem (remark (iii) of Section 7 of \cite{Bourgain-projection}).

\begin{theorem}[\cite{Bourgain-projection}]
\label{bourgain-projection-2}
Given $0<\al<2,\, \be>0$ and $\ka>0$, there exists $\tau_0>0$ and $\eta>\al/2$ such that the following holds.

Let $\mu$ be a probability measure on $\ZS^1$ such that (a non-concentration property)
\begin{equation*}
    \max_\theta \mu([\theta-\rho,\theta+\rho])<C\rho^\ka.
\end{equation*}
Let $\de\ll 1$ be sufficiently small and let $\sa\subset [1,2]\times [1,2]$ be a union of $\de$-squares satisfying
\begin{equation*}
    |\sa|=\de^{2-\al}
\end{equation*}
and 
\begin{equation}
\label{Section7-nonconcentration-1}
    \max_x|\sa\cap B(x,\rho)|<\rho^\be|\sa|\,\,\,\text{ for }\de<\rho<\de^{\tau_0}.
\end{equation}
Then there exists a subset $D\subset\supp(\mu_1)$ with $\mu_1(D)>1-\de^{\ka\tau_0/2}$ so that for any $\theta\in\supp(\mu_1)$ and any $\sa'\subset\sa$ with $|\sa'|\geq \de^\e|\sa|$, 
\begin{equation}
\label{big-projection}
    |\pi_\theta(\sa')|>\de^{1-\eta}.
\end{equation}

\end{theorem}

Use a projective transform to send the horizontal line $\{x=0\}$ to infinity. Apply a point-line duality for the pair $(P'', \bar T)$ to obtain 
\begin{enumerate}
    \item A point set $\wt P$ with $|\wt P|\lessapprox\De^{-2s}$.
    \item A $(\De,t-s,C)$-set $\Theta$ of direction so that for each $\theta\in\Theta$, there is a subset $\wt P'\subset\wt P$ with $|\wt P'|\gtrapprox|\wt P|$ and $|\pi_\theta(P')|_\De\lessapprox\De^{-s}$.
\end{enumerate}

\begin{figure}[ht]
\centering
\includegraphics[width=0.5\textwidth]{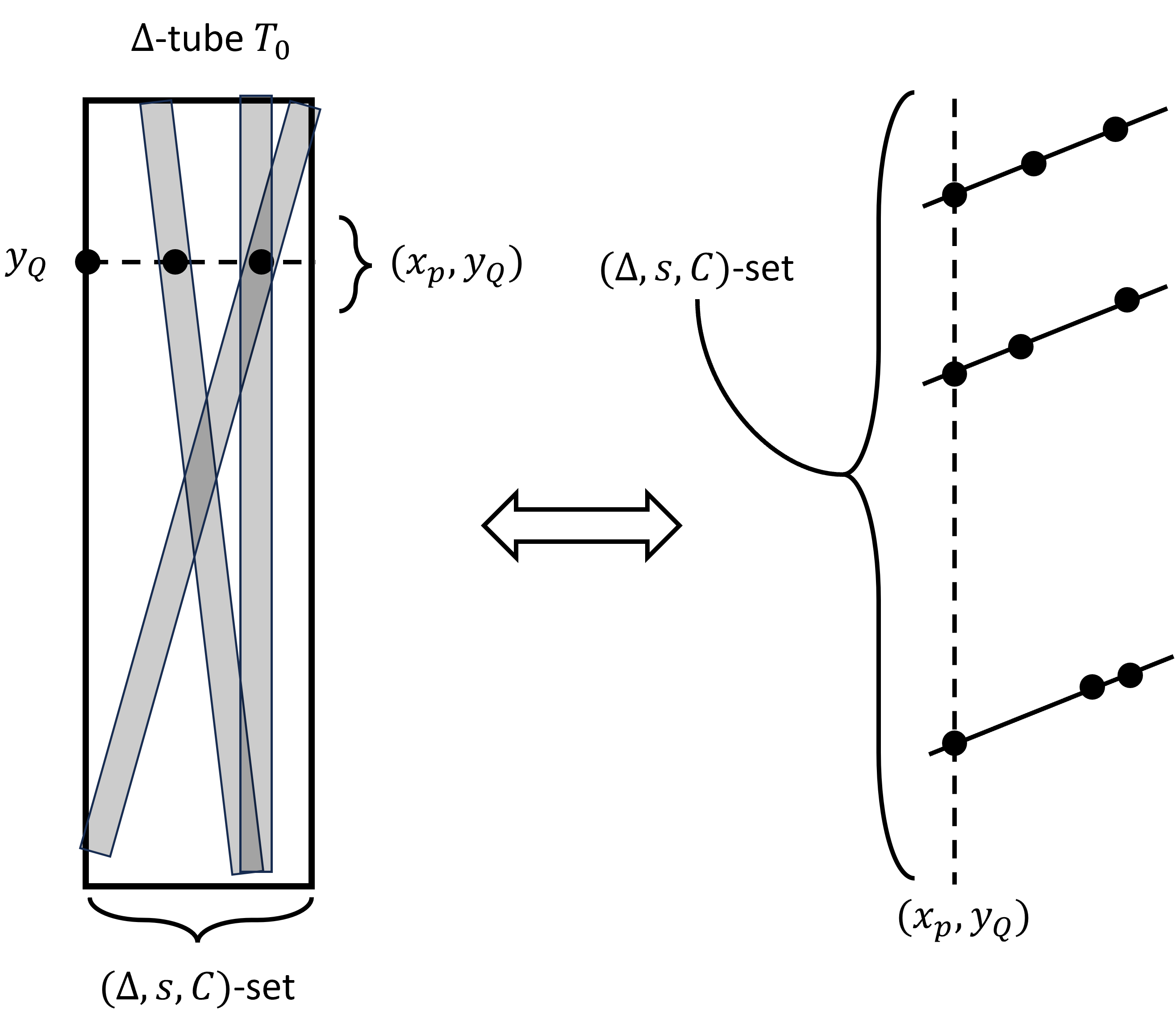}
\caption{Point-line duality: If $y_Q$ is fixed, $(x_p, y_Q)$ corresponds to parallel lines. Tubes $\cT(p)$ intersect $(x_p,y_Q)$ correspond to the points on the lines on the left.}
\end{figure}
We used here a slightly different point-line duality, which is defined by the map
\begin{equation*}
    \mathbf{D}:(a,b) \mapsto \{x=ay+b: y \in \mathbb{R}\},
\end{equation*}
while earlier in the paper the map $\mathbf{D} :(a,b) \mapsto \{ y =ax+b: x\in \mathbb{R}\} $ is used.

This contradicts Theorem \ref{bourgain-projection-2} with $\delta$ replaced by $ \De$ and $\alpha= 2+2s$. Thus, contrary to \eqref{negation-1}, we have either $|\cT|_\de\geq\de^{-2s-\e}$ or $|\cT|_{\de^{1/2}}\geq\de^{-s-\e}$.

\section{Acknowledgement}
The authors would like to express our gratitude to the organizers, Philip Gressman, Yumeng Ou, Hong Wang, and Josh Zahl, for organizing such an excellent ``Study Guide Writing Workshop 2023" at the University of Pennsylvania. In particular, we would like to thank Hong Wang as the mentor of our group. The conference was supported by NSF Grant DMS-2142221; the first author was additionally supported by NSF Grants DMS-2037851 and DMS-2246906; and the second author is supported by MOE Taiwan-Caltech Fellowship.

\bibliographystyle{alpha}
\bibliography{bibli}

\end{document}